\documentclass[10pt]{amsart}
\usepackage{graphicx} 
\usepackage{mystyle}
\usepackage{amssymb,amsthm,amsmath}
\usepackage[utf8]{inputenc}
\usepackage[english]{babel}
\usepackage{hyperref}
\usepackage{caption}
\usepackage[skip=0.5cm]{caption} 
\usepackage{tikz}
\usepackage{cite}
\usetikzlibrary{patterns}
\usetikzlibrary{decorations.pathreplacing}

\author{Janusz Ginster}
\address{Weierstrass Institute \\
Anton-Wilhelm-Amo-Str. 39 \\ 10117 Berlin \\ Germany }
\email{janusz.ginster@wias-berlin.de}

\newcommand{\eps}{\varepsilon}

\title[Formation of microstructure and vortices for an energy related to helimagnetism]{On the formation of microstructure and the occurrence of vortices in a singularly perturbed energy related to helimagnetism: a scaling law result}

\begin{document}

\begin{abstract}
    In this work, singularly perturbed energies arising from discrete $J_1$-$J_3$-models are studied. 
    The energies under consideration consist of a non-convex bulk term and a higher-order regularizing term and are subject to incompatible boundary conditions.
    In contrast to existing results in the literature, in this work, admissible fields are not necessarily gradient fields, instead their $\operatorname{curl}$ is linked to topological singularities, so-called vortices, in the discrete $J_1$-$J_3$-model.
    The main result of this work is a scaling law for the minimal energy with respect to three parameters: one measuring the incompatibility of the boundary conditions, the second measuring the strength of the regularizing term, and the third being related to the interatomic distance in the discrete model.
    The shown result implies in particular that in certain parameter regimes, minimizers necessarily develop vortices.
    A key tool in the analysis is a careful modification of the celebrated ball-construction technique that, due to a lack of rigidity, considers simultaneously both the bulk energy and the regularizing term. 
\end{abstract}


\maketitle

\section{Introduction}

We study several variants of variational energies of the form
\begin{equation}\label{eq: energy intro}
E_{\sigma}(\beta) = \int_{(0,1)^2} W(\beta) \, d\mathcal{L}^2 + \sigma |D\beta|((0,1)^2),
\end{equation}
where $\sigma > 0$, $\beta: (0,1)^2 \to \R^2$ and $W$ is a multi-well potential given by $W(\beta) = (1-\beta_1^2)^2 + (1-\beta_2^2)^2$.
Such models arise from a statistical-mechanical description of helimagnetic compounds and help explain pattern formation and the occurrence of certain topological defects of the magnetic spin field, so-called vortices, under incompatible boundary conditions, see Section \ref{sec: heuristics}.

Our main result will be a scaling law result for the infimal energies, which allows to determine parameter regimes for these energies in which uniform structures, finely oscillating patterns, or vortices appear to be energetically favorable. 

Establishing scaling laws has proven useful in a broad range of variational problems in which determining exact minimizers analytically or numerically is very challenging due to expected fine-scale structures of the minimizer, cf.~\cite{kohn:06}.
In such problems, the formation of patterns is often the result of the competition of a non-(quasi)convex part of the energy that allows for fine-scale oscillations and a term that penalizes non-uniform structures. 
A non-exhaustive list of results that successfully adopt this strategy includes \cite{B11-1994-KM,conti:00,conti:06,CapellaOtto2009,CapellaOtto2012,B11-chan-conti:14-1,B11-KKO,bella-goldman:15,conti-zwicknagl:16,ContiDiermeierMelchingZwicknagl2020,ruland2021energy,GiTrRuZw25} for martensitic microstructure, \cite{KoWi14,Kohn-Wirth:15} for compliance minimization,
 \cite{choksi-kohn:98,choksi-et-al:98,DKMO:06,otto-steiner:10,otto-viehmann:10,knuepfer-muratov:11,Dabade-et-al:19,DaVeJa20,DeKnOt06} for micromagnetism, \cite{CKO03,choksi-et-al:08,B11-ContiOttoSerfaty} for type-I-superconductors, \cite{belgacem-et-al:02,bella-kohn:14,BCM2017,CoDeMu05,GiNeZw24} for compressed thin elastic films, and \cite{ContiOrtiz:05,conti-zwicknagl:16} for dislocation patterns.

The work most closely related to this manuscript is the scaling law result in \cite{GinZwi:22} for an energy similar to  \eqref{eq: energy intro}  in the case that admissible fields $\beta$ are gradients.
It is shown for $\theta \in (0,1/2]$ and $\sigma > 0$ that
\begin{equation} \label{eq: result gradient}
\inf_{u(0,y) = (1-2\theta)y} E_{\sigma}(\nabla u) \sim \min\left\{ \theta^2, \sigma \left( \frac{|\log \sigma|}{|\log \theta|} + 1\right) \right\}.
\end{equation}
We note that for $\theta \in (0,1/2]$ the boundary condition $u(0,y) = (1-2\theta) y$ is incompatible with a constant gradient in the wells of $W$ given by $W^{-1}(0) = \left\{ \begin{pmatrix} 1 \\ 1 \end{pmatrix},  \begin{pmatrix} -1 \\ 1 \end{pmatrix},  \begin{pmatrix} 1 \\ -1 \end{pmatrix},  \begin{pmatrix} -1 \\ -1 \end{pmatrix} \right\}$. 
Hence, at least for small values of $\sigma > 0$ it can be expected that the boundary values are attained by fine oscillations of $\partial_y u$ close to $\{0\} \times (0,1)$.
Regarding upper bounds, the above result confirms this, as the term $\theta^2$ corresponds to the uniform structure $u(x,y) = (1-2\theta) y \pm 1$, whereas the logarithmic term can be proven by a self-similarly refining branching construction, see Figure \ref{fig: branching}.
Similar results in the gradient case are shown on different domains in \cite{Gi23} and for more general functionals in \cite{GinZwi:23}. 

In this work we consider the more general situation in which admissible fields $\beta$ are not necessarily $\operatorname{curl}$-free. 
Instead, we assume for $\eps > 0$ that $\operatorname{curl } \beta = \sigma \sum_{i} \gamma_i \delta_{x_i}*\rho_{\eps}$, where $\gamma_i \in \{\pm 1\}$, $x_i \in (0,1)^2$ and $\rho_{\eps}$ is a standard mollifier on scale $\eps$.
For an interpretation of this condition in the context of helimagnetic compounds, we refer to Section \ref{sec: heuristics}.

We show that for $\sigma \gtrsim \eps > 0$ and $\theta \in (0,1/2]$ it holds, cf.~Theorem \ref{thm: main},
\begin{equation}\label{eq: scaling law}
    \inf_{\beta_2(0,y) = 1-2\theta} E_{\sigma}(\beta) \sim \min\left\{ \theta^2, \sigma \left( \frac{|\log \sigma|}{|\log \theta|} + 1\right), \theta \frac{\sigma^3}{\eps^2} + \theta \sigma |\log \theta| \right\}.
\end{equation}
In addition to the two terms already present in \eqref{eq: result gradient}, an additional third term appears, corresponding to a configuration with vortices that act analogously to dislocations in solids at interfaces, see e.g., \cite{GiNeZw24}. 
Indeed, consider equidistant positive vortices ($\gamma_i = 1$ for all $i$) spaced at distance $\frac{\sigma}{2\theta}$ near the left boundary $\{0\} \times (0,1)$. In this case $\int_{ (0,1) \times (y_1,y_2)} \operatorname{curl } \beta \, dx \sim (y_2 - y_1) 2 \theta$. 
This allows to interpolate in a $\operatorname{curl}$-free way from $\beta \equiv (1,1)^T$ towards the boundary conditions at $\{0\} \times (0,1)$, see Figure \ref{fig: upper bound}. In order to estimate the energy of such a configuration, note that a field $\beta: \R^2 \to \R^2$ satisfying $\operatorname{curl } \beta = \sigma \delta_0 * \rho_{\eps}$ typically behaves like $|\beta| \sim \min\{ \frac{\sigma}{\eps}, \frac{\sigma}{|x|} \}$. 
Hence, we estimate the energy induced by the above configuration
\begin{align}
\int_{(0,1)^2} W(\beta) \, d\mathcal{L}^2 &\sim \frac{\theta}{\sigma} \left( \int_{B_{\eps}(0)} \frac{\sigma^4}{\eps^4} \, d \mathcal{L}^2 + \int_{B_{\sigma / \theta}(0) \setminus B_{\eps}(0)} \frac{\sigma^2}{|x|^2} \, d\mathcal{L}^2 \right) \\ & \sim \frac{\theta}{\sigma} \left( \frac{\sigma^4}{\eps^2} + \sigma^2 \log \left( \frac{\sigma}{\theta \eps} \right)  \right) \lesssim \theta \frac{\sigma^3}{\eps^2} + \theta \sigma |\log \theta|,
\end{align}
where we used that $\log (\sigma / (\theta \eps) ) \lesssim \log \frac1{\theta} + \frac{\sigma^2}{\eps^2}$.
Similarly, we expect for a field with a single vortex at the origin $|D\beta| \sim \min\{\frac{\sigma}{\eps^2},\frac{\sigma}{|x|^2}\}$ and therefore for the configuration above
\[
\sigma |D\beta| \sim  \sigma \frac{\theta}{\sigma} \left( \int_{B_{\eps}(0)} \frac{\sigma}{\eps^2} \, d\mathcal{L}^2 + \int_{B_{\sigma / \theta}(0) \setminus B_{\eps}(0)} \frac{\sigma}{|x|^2} \, d\mathcal{L}^2  \right) \sim \sigma \theta +  \sigma \theta \log\left( \frac{\sigma}{\theta \eps} \right) \lesssim \sigma \theta |\log \theta|.
\]
We will also consider a quadratic regularizing term (cf.~Section \ref{sec: main}) for which a similar computation for the discussed configuration shows
\[
\sigma^2 \int_{(0,1)^2} |D \beta|^2 \, d\mathcal{L}^2 \sim \sigma^2 \frac{\theta}{\sigma} \left( \int_{B_{\eps}(0)} \frac{\sigma^2}{\eps^4} \, d\mathcal{L}^2 + \int_{B_{\sigma / \theta}(0) \setminus B_{\eps}(0)} \frac{\sigma^2}{|x|^4} \, d\mathcal{L}^2  \right) \sim \theta \frac{\sigma^3}{\eps^2}.
\]
A matching lower bound has to be proven ansatz-free. 
The key difficulties in establishing the lower bound are twofold. 
First, it needs to be carefully shown that the techniques from the $\operatorname{curl}$-free setting (see \cite{GinZwi:22, GinZwi:23, Gi23}) can be modified to prove the same lower bounds as in the $\operatorname{curl}$-free setting for fields with not too many vortices. 
Second, in order to establish lower bounds for fields with a relevant number of vortices we distinguish the induced energy close and far from the vortices similar to the heuristic computation above for the upper bound. 
For the technically more demanding estimate far from the vortices, we adapt the celebrated ball-construction to obtain a logarithmic lower bound.
To the author's knowledge, the main difference to applications of this technique in the literature (in the case of Ginzburg-Landau \cite{Je99,Sa98} and dislocations \cite{dLGaPo12,Gi18,GiNeZw24}) is that in our setting the term of the energy $\int W(\beta) \, d\mathcal{L}^2$ does not allow for any rigidity in the sense that there may exist admissible fields $\beta$ that lie far from the vortices completely in the wells of $W$ but instead induce surface energy. 
For a heuristic explanation of the scales of the construction, we refer to \cite[Section 3.1]{GinZwi:22}.
Hence, the proof of the lower bound in our case has to compensate for this non-rigidity through the regularizing higher-order term, see Proposition \ref{prop: ball constr}.

\subsection{Heuristic connection to a frustrated spin system and interpretation of the result}\label{sec: heuristics}

Let us briefly explain the heuristic connection of the continuum energy discussed above, \eqref{eq: energy intro}, to a discrete $J_1 - J_3$ model on the square lattice $\eps \Z^2$. The arguments presented below can  essentially be found in \cite{CiSo15, CiFoOr19, GiKoZw}.

For $\eps,\alpha > 0$ assign to $u: \eps \Z^2 \cap [0,1)^2 \to S^1$ the (renormalized) interaction energy
\[
F_{\alpha,\eps}(u) = -\alpha \sum_{k,j \in [0,1)^2 \cap \eps\Z^2: |k-j| = \eps} u(k) \cdot u(j) + \sum_{k,j \in [0,1)^2 \cap \eps\Z^2: |k-j| = 2\eps} u(k) \cdot u(j),
\]
i.e., ferromagnetic interactions between nearest neighbors and anti-ferromagnetic interactions between second neighbors in rows and columns (\textit{not} next-to-nearest, i.e.~diagonal, interactions, under whose influence the behavior of the energy changes drastically, see \cite{CiFoOr22}). 
The two competing interactions of the energy cannot be minimized simultaneously, leading to frustration in the discrete system. 
However, up to boundary effects, the energy can be rewritten as
\begin{align}
I_{\alpha,\eps}(u) \approx &\frac12 \sum_{j \in [0,1)^2 \cap \eps \Z^2} \left( \left|  u(j) - \frac{\alpha}2 u(j + \eps e_1) + u(j+2 \eps e_1) \right|^2  - 2 - \frac{\alpha^2}4 \right) \\ &+ \frac12 \sum_{j \in [0,1)^2 \cap \eps \Z^2} \left( \left|  u(j) - \frac{\alpha}2 u(j + \eps e_2) + u(j+2 \eps e_2) \right|^2  - 2 - \frac{\alpha^2}4 \right).
\end{align}
In particular, it can be seen that for $\alpha > 4$ the right-hand side is minimized by constant spin fields, whereas for $0<\alpha <4$ optimal configurations rotate with an optimal angle $\pm \arccos(\alpha/4)$ in rows and columns, see Figure \ref{fig: min discrete energy}.  
\begin{figure}
\centering
\begin{tikzpicture}[scale = 3]
    \draw[->] (0,0) -- (0,3/2);
    \draw[->] (0,0) -- (-1/3,3/4);
    \draw[->] (0,0) -- (1/3,3/4);
    \draw[dotted,->] (-1/3,3/4) -- (0,3/2);

    \draw(0,3/2) node[anchor=west] {$\frac{\alpha}4 u(j+e_1)$};
    \draw(1/3,3/4) node[anchor=west] {$ u(j+e_2)$};
    \draw(-1/3,3/4) node[anchor=east] {$u(j)$};    
\end{tikzpicture} \qquad \qquad 
\begin{tikzpicture}[scale=0.6]
    \foreach \i in {1,...,10}{
    \foreach \j in {1,...,10}{
    \fill[black] (\i,\j) circle (1pt);
    \draw[->] (\i,\j) -- +(\i*20 + \j*20:10pt);
    }
    }
\end{tikzpicture}
    \caption{Left: Sketch of neighboring spins such that $u(j) - \frac{\alpha}2 u(j+e_1) + u(j+e_2) = 0$. The angle between these neighboring spins is given by $\arccos(\alpha/4)$. Right: Sketch of a ground state of the discrete energy in which the spin field (the sketched spin field is scaled to length of order $\eps$ in order to fit the picture) rotates counter-clockwise in rows and columns. }
    \label{fig: min discrete energy}
\end{figure}
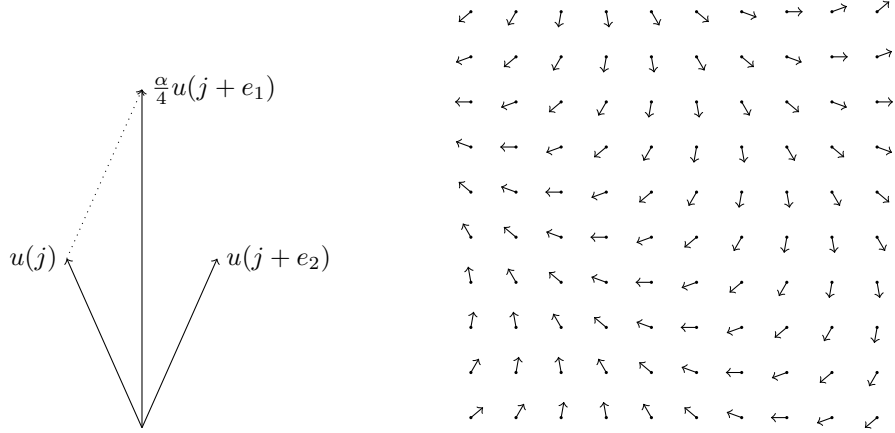
We are interested in the case $0< \alpha < 4$. Let us now define the vertical and horizontal angular change as
\begin{align}
\theta^{hor}(j)& = \operatorname{sign} (u(j) \times u(j+\eps e_1)) \arccos\left( u(j) \cdot u(j+\eps e_1) \right) \\ \text{ and } \qquad \theta^{ver}(j) &= \operatorname{sign}(u(j) \times u(j+\eps e_2))  \arccos\left(  u(j) \cdot u(j+ \eps e_2) \right),
\end{align}
where $\arccos:[-1,1] \to [0,\pi]$ denotes the usual inverse to $\cos$. 
With this definition it follows that 
\begin{equation}\label{eq: angular velocity}
u(j+ \eps e_1) = e^{i\theta^{hor}(j)} u(j) \text{ and } u(j+\eps e_2) = e^{i\theta^{ver}(j)} u(j).
\end{equation}
Moreover, note that the field $(\theta^{hor}, \theta^{ver})$ satisfies the discrete $\operatorname{curl}$-condition
\begin{equation}\label{eq: discrete curl}
    \theta^{hor}(j) + \theta^{ver}(j+e_1) - \theta^{hor}(j+e_2) - \theta^{ver}(j)  \in \{-2\pi,0,2\pi\}.
\end{equation}
In case the above expression lies in $\{\pm 2\pi\}$ we call this a \textit{vortex} of the spin field, i.e.~the spin field rotates clockwise or counter-clockwise around an $\eps$-cell.
In the following we will identify the fields $u$, $\theta^{hor}$ and $\theta^{ver}$ with functions defined on $(0,1)^2$ by interpolation.
Let us now write for $\delta = \frac{4 - \alpha}4$, i.e., $\alpha = 4(1-\delta)$, 
\begin{align}
&    \eps^2 \sum_{j \in [0,1)^2 \cap \eps \Z^2} \left|  u(j) - \frac{\alpha}2 u(j + \eps e_1) + u(j+2 \eps e_1) \right|^2 \\
= & \eps^2 \sum_{j \in [0,1)^2 \cap \eps \Z^2} \left|  \underbrace{u(j) - 2 u(j + \eps e_1) + u(j+2 \eps e_1)}_{\approx \eps^2 \partial_1 \partial_1 u} + 2 \delta u(j+\eps e_1) \right|^2 \\
\approx &\int_{(0,1)^2} \eps^4 |\partial_1 \partial_1 u|^2 + 4\delta \eps^2 u\cdot (\partial_1 \partial_1 u) + 4 \delta^2 \, d\mathcal{L}^2.
\end{align}
From \eqref{eq: angular velocity} we obtain (we identify $\mathbb{C} \simeq \R^2$)
\[
\partial_1 \partial_1 u \approx \partial_1 \left( i \frac{\theta^{hor}}{\eps} u \right) \approx i \left( \partial_1  \frac{\theta^{hor}}{\eps} \right) u - \left( \frac{\theta^{hor}}{\eps} \right)^2 u.
\]
In particular, it follows 
\[
|\partial_1 \partial_1 u|^2 \approx \left(\partial_1 \frac{\theta^{hor}}{\eps} \right)^2 + \left(\frac{\theta^{hor}}{\eps} \right)^4.
\]
Then we find from the above
\begin{align}
    &\int_{(0,1)^2} \eps^4 |\partial_1 \partial_1 u|^2 + 4\delta  \eps^2 u\cdot (\partial_1 \partial_1 u) + 4 \delta^2 \, d\mathcal{L}^2 \\
   \approx  &\int_{(0,1)^2} \eps^4 \left(\partial_1  \frac{\theta^{hor}}{\eps} \right)^2 + \eps^4 \left( \frac{\theta^{hor}}{\eps} \right)^4 - 4 \delta  \eps^2 \left(\frac{\theta^{hor}}{\eps} \right)^2 + 4 \delta^2 \, d\mathcal{L}^2 \\
   = &\int_{(0,1)^2} \eps^2 \left(\partial_1  \theta^{hor} \right)^2 + \left( \left( \theta^{hor} \right)^2 - 2 \delta \right)^2 \, d\mathcal{L}^2 \\
   = & 4 \delta^2 \int_{(0,1)^2} \frac{\eps^2}{2\delta} \left(  \partial_1 \frac{1}{\sqrt{2\delta}} \theta^{hor}\right)^2 + \left( \left( \frac{1}{\sqrt{2 \delta}} \theta^{hor}\right)^2 - 1\right)^2 \, d\mathcal{L}^2.
\end{align}
Therefore we find for $\sigma = \frac{\eps}{\sqrt{2\delta}}$, $\beta_1 = \frac{1}{\sqrt{2\delta}} \theta^{hor}$ and $\beta_2 = \frac{1}{\sqrt{2\delta}} \theta^{ver}$ that
\[
\eps^2 \left( I_{\alpha,\eps}(u) - \min I_{\alpha,\eps} \right) \approx 2 \delta^2 \int_{(0,1)^2} \sigma^2 (\partial_1  \beta_1)^2 + \sigma^2 (\partial_2  \beta_2)^2 + \left( \beta_1^2 - 1 \right)^2 + \left( \beta_2^2 - 1 \right)^2 \, d\mathcal{L}^2.
\]
This is a version of the energy \eqref{eq: energy intro}.
Eventually, we note that the discrete condition \eqref{eq: discrete curl} translates for $\beta$ to a circulation condition of the form
\[
\int_{\partial \left( (x_1,x_1 + \eps) \times (x_2,x_2+\eps) \right)}
\beta \cdot \tau \, d\mathcal{H}^1 \in \sigma \left\{ -2\pi,0,2\pi \right\},
\]
where $x \in (0,1)^2$ and $\tau \in S^1$ is the positively oriented unit tangent. We will interpret this condition in the following as $\operatorname{curl } \beta = 2\pi \sigma \sum_{k} \gamma_k \delta_{x_k} * \rho_{\eps}$, where $x_k \in (0,1)^2$ and $\rho_{\eps}$ is a mollifier on scale $\eps$.

Next, we comment on the considered boundary conditions for $\beta$, i.e., $\beta_2(0,y) = 1-2 \theta$, where $\theta \in (0,1/2]$. Following the heuristic argument above, the boundary conditions translate into spin fields $u$ that rotate along the column $\{0\} \times \left([0,1) \cap \eps \Z \right) $ with the angular velocity $\theta^{ver} = (1-2\theta) \sqrt{2\delta} \approx (1-2\theta)  \arccos(1-\delta)$, where we used that $1 - \delta = \cos(\arccos(1-\delta)) \approx 1 - \frac{\arccos(1-\delta)^2}{2}$ implies that $\arccos(1-\delta) \approx \sqrt{2\delta}$. In particular, the prescribed angular velocity on the boundary is suboptimal, and we study in this work how this leads to the formation of patterns of the four different ground states (rotations with optimal angular velocity either clockwise or counter-clockwise in rows and columns) or the occurrence of vortices.

We close this section by interpreting the scaling law \eqref{eq: scaling law}  in the context of the discrete energy $F_{\alpha,\eps}$. 
As discussed below \eqref{eq: scaling law}, the third term of the scaling law corresponds to a regime in which vortices are present in optimal configurations.
Moreover, for $\theta = 1/2$ this term reads as 
\[
\theta \frac{\sigma^3}{\eps^2} + \theta \sigma |\log \theta| \sim \frac{\sigma^3}{\eps^2} + \sigma \sim \frac{\eps}{\delta^{3/2}}. 
\]
This is exactly the energy that was conjectured in \cite{GiKoZw} to be the energy scaling in the regime with vortices for the discrete energy. However, only an upper bound construction could be provided.
It is to be expected that the proof technique developed in this work can be modified to prove the full scaling law also in the discrete setting, at least for $\theta = 1/2$.

\subsection{Outline}

In the next section, we will introduce the mathematical setting and fix the notation of this paper. Moreover, we will present the main scaling law result, Theorem \ref{thm: main}.  In Section \ref{sec: prelims} we collect some preliminary results. Section \ref{sec: upper} is devoted to proving the upper bound of Theorem \ref{thm: main}, whereas the lower bound of Theorem \ref{thm: main} will be proven in Section \ref{sec: lower} 

\section{The mathematical setting and main result}\label{sec: main}

Let $\eps, \sigma, \theta > 0$ and $\Omega = (0,1)^2$. 
We start with defining the set of admissible vorticity measures
\begin{equation}\label{eq: vorticity measure}
    \mathcal{M}_{\sigma,\eps} := \left\{ \mu \in \mathcal{M}(\Omega): \mu = \sigma\sum_{i=1}^n \gamma_i  \, \delta_{x_i}*\rho_{\varepsilon},  n\in \N, \gamma_i \in \{\pm 1\}, \,  B_{\eps}(x_i) \subseteq \Omega \text{ and } B_{\varepsilon}(x_i) \cap B_{\eps}(x_j) = \emptyset \right\},
\end{equation}
where $\rho \in C^{\infty}_c(B_{\frac14}(0);[0,\infty))$ with $\int_{B_{1/4}(0)} \rho \, d\mathcal{L}^2 = 1$ is fixed  and $\rho_{\eps}(x) = \eps^{-2} \rho(x/\eps)$.
Next, we define the sets of admissible functions
\begin{align}\label{eq: admissible functions}
\mathcal{A}^{(1)}_{\sigma,\theta,\eps} &:= \left\{ \beta \in L^4(\Omega;\R^2): D\beta \in \mathcal{M}(\Omega;\R^{2\times 2}), \, \beta_2(0,\cdot) = 1-2\theta \text{ and } \operatorname{curl} \beta \in \mathcal{M}_{\sigma,\eps} \right\}, \\
\mathcal{A}^{(2)}_{\sigma,\theta,\eps} &:= \left\{ \beta \in L^4(\Omega;\R^2): D\beta \in L^2(\Omega;\R^{2\times 2}), \, \beta_2(0,\cdot) = 1-2\theta \text{ and } \operatorname{curl} \beta \in \mathcal{M}_{\sigma,\eps} \right\}, \\
\mathcal{A}^{(a)}_{\sigma,\theta,\eps} &:= \left\{ \beta \in L^4(\Omega;\R^2): \partial_1\beta_1, \partial_2 \beta_2 \in L^2(\Omega;\R^{2\times 2}), \, \beta_2(0,\cdot) = 1-2\theta \text{ and } \operatorname{curl} \beta \in \mathcal{M}_{\sigma,\eps} \right\}.
\end{align}
Note that the boundary value for $\beta_2$ in the sets $\mathcal{A}^{(1)}_{\sigma,\theta,\eps}$ and $\mathcal{A}^{(2)}_{\sigma,\theta,\eps}$ can be understood in the sense of traces for the spaces $BV((0,1);\R^2)$ and $W^{1,2}((0,1)^2;\R^2)$, respectively. In the set $\mathcal{A}^{(a)}_{\sigma,\theta,\eps}$ the boundary values can be understood in the sense of traces for $L^2$-functions whose $\operatorname{curl }$ is also in $L^2$, see \cite[Chapter IX., Part A, Theorem 2]{Lions3} or \cite[Chapter 4]{Boyer/Fabrie:2012}. 

Next, we define the energies $E^{(1)}_{\sigma,\theta,\eps}: \mathcal{A}_{\sigma,\theta,\eps}^{(1)} \to [0,\infty]$,  $E^{(2)}_{\sigma,\theta,\eps}: \mathcal{A}_{\sigma,\theta,\eps}^{(2)} \to [0,\infty]$ and $E^{(a)}_{\sigma,\theta,\eps}: \mathcal{A}_{\sigma,\theta,\eps}^{(a)} \to [0,\infty]$ as
\begin{align}\label{eq: def energy}
E^{(1)}_{\sigma,\theta,\eps}(\beta) &= \int_{\Omega} W(\beta) \, d\mathcal{L}^2 + \sigma |D \beta|(\Omega), \\
E^{(2)}_{\sigma,\theta,\eps}(\beta) &= \int_{\Omega} W(\beta) + \sigma^2 |D \beta|^2 \, d\mathcal{L}^2,\\
E^{(a)}_{\sigma,\theta,\eps}(\beta) &= \int_{\Omega} W(\beta) + \sigma^2 (\partial_1 \beta_1)^2 + \sigma^2 (\partial_2 \beta_2)^2 \, d\mathcal{L}^2,
\end{align}
where $W(\beta) = (1-\beta_1^2)^2 + (1-\beta_2^2)^2$.

The main theorem is the following scaling law for the infimal energies.

\begin{theorem}\label{thm: main}
    There exist constants $C \geq c > 0$ such that it holds for all $\sigma > \sqrt{2} \pi \eps >0$, $\theta \in (0,1/2)$ and $\epsilon \in \{1,2,a\}$ that 
    \[
    c \cdot s(\sigma, \eps, \theta) \leq \inf_{\beta \in \mathcal{A}_{\sigma,\theta,\eps}^{(\epsilon)}} E^{(\epsilon)}_{\sigma, \theta, \eps} \leq C \cdot s(\sigma, \eps, \theta),
    \]
    where $s(\sigma, \eps, \theta) = \min\left\{ \theta^2, \sigma \left( \frac{|\log \sigma|}{|\log \theta|} + 1 \right), \theta \frac{\sigma^3}{\eps^2} + \theta \sigma \log\left( \frac{\sigma}{\eps \theta} \right) \right\}$.
\end{theorem}
\begin{proof}
    The upper bound follows from Proposition \ref{prop: upper} whereas the lower bound follows from Proposition \ref{prop: main lower bound} and Proposition \ref{prop: lower aniso}.
\end{proof}

\subsection{Notation}

Throughout the text, we denote by $c$ and $C$ generic constants that may change from expression to expression and do not depend on the problem parameters.
Moreover, we will identify $\mathbb{C}$ with $\R^2$ and denote by $e_1$ and $e_2$ the two canonical basis vectors for $\R^2$. 
In the absence of ambiguities, we will not distinguish between row and column vectors.

For a measurable set $B\subseteq \R^n$ with $n=1,2$, we use the notation $| B| $ or $\calL^n(B)$ to denote its $n$-dimensional Lebesgue measure. 

The set of preferred values for $\beta$, $K \subseteq \R^2$, is defined as 
\[
K = \left\{ \begin{pmatrix} 1 \\ 1 \end{pmatrix}, \begin{pmatrix} 1 \\ -1 \end{pmatrix}, \begin{pmatrix} -1 \\ 1 \end{pmatrix}, \begin{pmatrix} -1 \\ -1 \end{pmatrix}   \right\}.
\]
Throughout this manuscript, we will occasionally write for $\xi \in \R$ the expression $|\xi \pm 1|$ instead of $\min\{ |\xi - 1|, |\xi + 1|\}$. 

For $B \subseteq \R^2$ open and $\beta \in L^4(B;\R^2) \cap BV(B;\R^2)$, we use the notation $E_{\sigma,\theta,\eps}^{(1)}(\beta;B)$ for the energy on $B$, i.e.,
\begin{align}\label{eq:restrenergy}
&E_{\sigma,\theta,\eps}^{(1)}(\beta;B) := 
\int_{B} W(\beta) \, d\mathcal{L}^2 + \sigma | D \beta| (B). 
\end{align}
In addition for $x \in (0,1)$ and $I\subseteq (0,1)$ Lebesgue-measurable, we write for $\beta \in \mathcal{A}^{(1)}_{\sigma,\theta,\eps}$
\[
E_{\sigma,\theta,\eps}^{(1)}(\beta;\{x\} \times I) = \int_{I} W(\beta(x,y)) dy + |\partial_1 \beta(x,\cdot)|(I).
\]
Note that since $\beta \in BV((0,1)^2)$ this formula makes sense for almost every $x \in (0,1)$ in the sense of slicing of $BV$-functions, see \cite{ambrosio-et-al:00}. Similarly, we write for $y \in (0,1)$ and $\beta \in \mathcal{A}^{(1)}_{\sigma,\theta,\eps}$
\[
E_{\sigma,\theta,\eps}(\beta;I \times \{y\}) := \int_{I} W(\beta(x,y)) \, dx + |\partial_2 \beta(\cdot,y)|(I).
\]
We use the analogous notation also for the energies $E^{(2)}_{\sigma,\theta,\eps}$ and $E^{(a)}_{\sigma,\theta,\eps}$.

Eventually, we will denote with $\Phi: \R \to \R$  the function which is uniquely determined by 
\begin{equation}\label{eq: def phi}
\Phi'(t) = |1-t^2|  \qquad \text{ and } \Phi(0) = 0. 
\end{equation}

\section{Preliminaries}\label{sec: prelims}

We start by proving that $W$ essentially behaves like the function $\max\{ \operatorname{dist}(\beta,K)^4, \operatorname{dist}(\beta,K)^2 \}$.

\begin{lemma}\label{lemma: est W}
    It holds for $W: \R^2 \to [0,\infty)$, $W(\beta) = (1-\beta_1^2)^2 + (1-\beta_2^2)^2$, that
    \[
    \max\{ \frac12 \operatorname{dist}(\beta,K)^4, \operatorname{dist}(\beta,K)^2 \} \leq W(\beta) \leq 18 \max\{ \operatorname{dist}(\beta,K)^4, \operatorname{dist}(\beta,K)^2 \} \text{ for all } \beta \in \R^2.
    \]
\end{lemma}
\begin{proof}
First, let $s \geq 0$. 
Then
\[
(1-s^2)^2 = (1-s)^2 (1+s)^2 \geq \max\left\{ (1-s)^4, (1-s)^2 \right\}.
\]
Therefore, it holds for $\beta \in \R^2$ that
\begin{align}
    W(\beta) &= (1-\beta_1^2)^2 + (1-\beta_2^2)^2 \\ 
    &\geq \max\left\{ \min |\beta_1 \pm 1|^2 + \min |\beta_2 \pm 1|^2, \min |\beta_1 \pm 1|^4 + \min |\beta_2 \pm 1|^4  \right\} \\
    &\geq  \max\bigg\{ \min |\beta_1 \pm 1|^2 + \min |\beta_2 \pm 1|^2, \\ &  \qquad \frac12 \left(\min |\beta_1 \pm 1|^4 + 2 \left(\min |\beta_2 \pm 1|^2\right) \left(\min |\beta_2 \pm 1|^2 \right) + \min |\beta_2 \pm 1|^4 \right) \bigg\} \\
    &= \max\left\{ \operatorname{dist}(\beta,K)^2,\frac12  \operatorname{dist}(\beta,K)^4  \right\}.
    \end{align}
    This shows the claimed lower bound.
    
Next, note that it holds for $s \geq 2$ that $(1+s)^2 = (s - 1 + 2)^2 = (s-1)^2 + 4 (s-1) + 4 \leq 9(s-1)^2$. Consequently, we find for $s \geq 0$
\[
(1-s^2)^2 = (1-s)^2 (1+s)^2 \leq 9 (1-s)^2 \mathbf{1}_{\{0 \leq s < 2\}} +  9 (1-s)^4 \mathbf{1}_{\{s > 2\}} \leq 9 \max\{ (s-1)^2, (s-1)^4\}.
\]
Hence, it follows by symmetry for $\beta \in \R^2$
    \begin{align}
    W(\beta) &= (1-\beta_1^2)^2 + (1-\beta_2^2)^2 \\ 
    &\leq 18 \max\left\{ \min |\beta_1 \pm 1|^2 +  \min |\beta_2 \pm 1|^2, \min |\beta_1 \pm 1|^4 + \min |\beta_2 \pm 1|^4\right\} \\
    &\leq 18 \max\left\{ \operatorname{dist}(\beta,K)^2, \operatorname{dist}(\beta,K)^4\right\}.
    \end{align}
\end{proof}

Similarly, we show some growth estimates for the function $\Phi$, cf.~\eqref{eq: def phi}.

\begin{lemma}\label{lem: properties phi}
    It holds for all $x,y \in \R$ that
    \[
    \frac18 |x-y|^2 \leq |\Phi(y) - \Phi(x)| \text{ and } |\Phi(x) - \Phi(a)| \leq 4 \left( |x - a| + |x - a|^3 \right) \text{ for } a \in \{\pm 1\}.
    \]
\end{lemma}
\begin{proof}
    Let $x,y \in \R$ with $x \leq y$. We distinguish the following cases:
    \begin{enumerate}
        \item Assume $0\leq x \leq y \leq 1$. Then we estimate
        \begin{align}
        |\Phi(y) - \Phi(x)| = \int_x^y 1-t^2 \, dt \geq \int_x^y 1-t \, dt = (y-x) - \frac12 (y^2 - x^2) &= \frac12 (y-x) \left( 2 - x - y \right) \\ & \geq \frac12 (y-x)^2,
        \end{align}
        since $y \leq 1$.
        \item Assume $1 \leq x \leq y$. Then we estimate
        \begin{align}
            |\Phi(y) - \Phi(x)| \geq  \int_x^y t - 1 \, dt = \frac12 (y-x) \left( y + x - 2 \right) \geq \frac12 (y-x)^2,
        \end{align}
        since $x \geq 1$.
        \item Assume $0\leq x \leq 1 \leq y$. Then we estimate using the monotonicity of $\Phi$, (1), (2) and the convexity of quadratic functions
        \begin{align}
            |\Phi(y) - \Phi(x)| = |\Phi(y) - \Phi(1)| + |\Phi(1) - \Phi(x)| \geq \frac12 (y-1)^2 + \frac12 (1-x)^2 \geq  \frac14 (y-x)^2.
        \end{align}
        \item Eventually, assume $x \leq 0 \leq y$. Then we estimate using (1),(2), (3), the symmetry of these estimates and the convexity of quadratic functions
        \begin{align}
            |\Phi(y) - \Phi(x)| = |\Phi(y) - \Phi(0)| - |\Phi(0) - \Phi(x)| \geq \frac14 y^2 + \frac14 (-x)^2 \geq \frac12 \left( \frac12 y - \frac12 x \right)^2 = \frac18 (y-x)^2.
        \end{align}
    \end{enumerate}
            By symmetry, this concludes the proof of the first inequality.
            
    For the second inequality, we estimate for $x \in \R$ and $a \in \{\pm 1\}$
    \begin{align}
        |\Phi(x) - \Phi(a)| &= \left| \int_x^a |1-t^2| \, dt \right| \\
        &\leq \left|\int_x^a 1 + t^2 \, dt \right| = |a-x| + \frac13 \left| a^3 - x^3 \right| \leq |a-x| + \frac13 |a-x| \left( a^2 + |x| |a| + x^2 \right).
    \end{align}
    Now we distinguish two cases. If $|x| \geq 2$ then it holds $|a| \leq |x| \leq |x-a| + |a| \leq 2 |x-a|$ and therefore
    \[
    |\Phi(x) - \Phi(a)| \leq |a-x| + \frac13 |a-x| \left( a^2 + |x| |a| + x^2 \right) \leq |a-x| + 4 |a-x|^3.
    \]
    On the other hand, if $|x| \leq 2$ then it holds 
    \[
    |\Phi(x) - \Phi(a)| \leq |a-x| + \frac13 |a-x| \left( a^2 + |x| |a| + x^2 \right) \leq (1 + 7/3) |a-x| \leq 4 |a-x|.
    \]
\end{proof}

We finish this section with a simple lemma to rewrite certain integrals over $\beta$.

\begin{lemma}\label{lemma: rewrite}
    Let $\beta \in \mathcal{A}^{(\epsilon)}_{\sigma,\theta,\eps}$ for $\epsilon \in \{\pm 1\}$. Then it holds for a.e.~$x \in (0,1)$ and a.e.~$0 < y_1 < y_2 <1$ that
    \begin{align}
        &\int_{y_1}^{y_2} \beta_2(x,t) \, dt - (1-2\theta) (y_2-y_1) \\
        = &\int_{0}^x \beta_1(s,y_2) - \beta_1(s,y_1) \, ds + \int_{(0,x) \times (y_1,y_2)} \operatorname{curl } \beta \, d\mathcal{L}^2.
    \end{align}
\end{lemma}
\begin{proof}
    By a convolution argument, we may assume that $\beta$ is smooth. Then we compute
    \begin{align}
        &\int_{y_1}^{y_2} \beta_2(x,t) \, dt - (1-2\theta) (y_2-y_1) \\
        =& \int_{0}^x \int_{y_1}^{y_2} \partial_1 \beta_2(s,t) \, ds dt \\
        =& \int_0^x \int_{y_1}^{y_2} \partial_2 \beta_1(s,t) \, ds dt + \int_{0}^x \int_{y_1}^{y_2} \operatorname{curl } \beta(s,t) \, ds dt  \\
        =& \int_{0}^x \beta_1(s,y_2) - \beta_1(s,y_1) \, ds + \int_{(0,x) \times (y_1,y_2)} \operatorname{curl } \beta \, d\mathcal{L}^2.
    \end{align}
\end{proof}

\section{Upper Bound}\label{sec: upper}

In this section we prove the upper bound of Theorem \ref{thm: main}.

\begin{proposition}\label{prop: upper}
    There exists $C>0$ such that it holds for all $\sigma > \eps > 0$, $\theta \in (0,1/2)$ and $\epsilon \in \{1,2,a\}$ that
    \[
    \inf E^{(\epsilon)}_{\sigma, \theta, \eps} \leq C \min\left\{ \theta^2, \sigma \left( \frac{|\log \sigma|}{|\log \theta|} + 1 \right), \theta \frac{\sigma^3}{\eps^2} + \theta \sigma |\log \theta |  \right\}.
    \]
\end{proposition}
\begin{proof}
    The upper bound
    \[
    \inf_{\beta_2(0,\cdot) = 1-2\theta} E^{(\epsilon)}_{\sigma, \eps}(\beta) \leq C \min\left\{ \theta^2, \sigma \left( \frac{|\log \sigma|}{|\log \theta|} + 1 \right)\right\}
    \]
    is shown in \cite[Theorem 1 and Remark 2]{GinZwi:22}. 
    Indeed, the first term is realized by the competitor $\beta(x,y) = e_1 + (1-2\theta) e_2$, whereas the second term is realized by a self-similar branching construction using the four preferred values in $K$ in a simplified convex integration setting, see Figure \ref{fig: branching} for an illustration. 

    \begin{figure}
        \centering
        \includegraphics[width=0.32\linewidth]{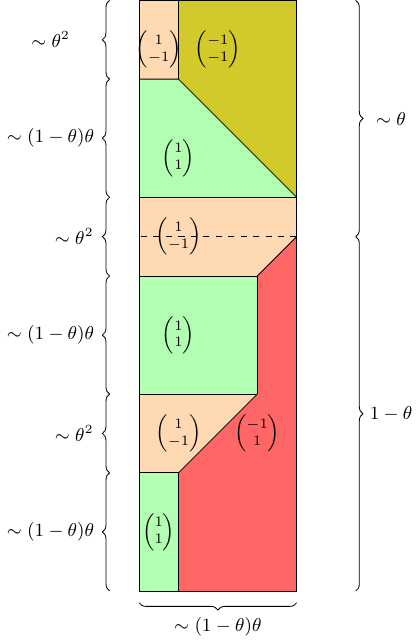} \qquad \qquad 
                \includegraphics[width=0.45\linewidth]{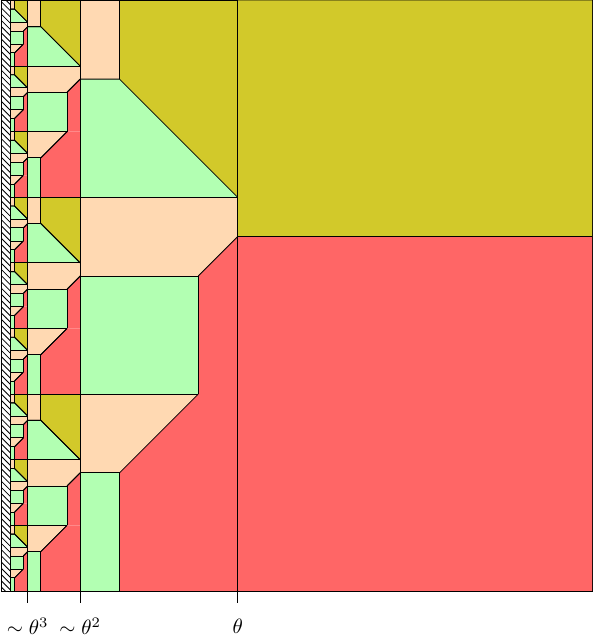}
        \caption{Sketch of the branching construction from \cite{GinZwi:22}. The different colors code the regions in which the values of the constructed gradient field $\nabla u$ are constant and in the set $K$. In the striped region the constructed function is interpolated towards the boundary so that $u$ satisfies $u(0,y) = (1-2\theta)y$.  Left: sketch of the building block; one oscillation of $\partial_y u$ is refined into $\sim \theta^{-1}$ oscillations of $\partial_y u$. Right: branching construction using the isotropically rescaled building blocks so that $\partial_y u$ oscillates more and more towards the boundary. Close to the boundary it holds in a weak sense $\partial_y u \approx 1-2\theta$.}
        \label{fig: branching}
    \end{figure}
    
    Hence, it remains to construct a competitor $\beta \in \mathcal{A}^{(\epsilon)}_{\sigma,\theta,\eps}$ such that $E^{(\epsilon)}_{\sigma,\theta,\eps}(\beta) \leq C \left( \theta \frac{\sigma^3}{\eps^2} + \theta \sigma \log\left( \frac{\sigma}{\theta \eps} \right) \right).$
    Note that when $\theta \leq \sigma$ it holds since $\eps < \sigma$ that $\theta^2 \leq \theta \sigma \leq \theta \frac{\sigma^3}{\eps^2}$ . 
    Therefore, it suffices to consider the case $\theta > \sigma > 0$.

    We start by defining a function $u: \R \times \left(0,\frac{\sigma}{2\theta}\right) \to \R$ as follows (see Figure \ref{fig: upper bound})
    \[
    u(x,y) = \begin{cases} 
    (1-2\theta)y + x &\text{ if } y \geq x, \\
    (1-2\theta) y + 2\theta \frac{x-y}{\frac{\sigma}{2\theta} - y} y + x &\text{ if } y<x< \frac{\sigma}{2\theta}, \\
     y + x &\text{ if } x \geq \frac{\sigma}{2\theta}.
    \end{cases}
    \]
\begin{figure}

\begin{tikzpicture}[scale =1.5]
    \fill[pattern= horizontal lines] (0,0) -- (4,4) -- (4,0) -- (0,0);
    \fill[green!30!white] (4,0) -- (4,4) -- (5,4) -- (5,0) -- (4,0);
    \fill[blue!30!white] (0,0) -- (0,4) -- (4,4) -- (0,0);
    
     \draw[thick] (0,0) -- (4,4) -- (4,0) -- (0,0);
     \draw[thick] (0,0) -- (5,0) -- (5,4) -- (0,4) -- (0,0);

    \draw [decorate,decoration={brace,amplitude=8pt,mirror}]
  (5.2,0) -- (5.2,4) node[midway, anchor=west, xshift = 0.3cm]{$\frac{\sigma}{2\theta}$};

   \draw [decorate,decoration={brace,amplitude=8pt,mirror}]
  (0,-0.2) -- (4,-0.2) node[midway, anchor=north, yshift = -0.3cm]{$\frac{\sigma}{2\theta}$};

     \draw node at (1,2.5) {$\begin{pmatrix} 1 \\ 1-2\theta \end{pmatrix}$};
    \draw node at (4.5,2) {$\begin{pmatrix} 1 \\ 1 \end{pmatrix}$};
    \end{tikzpicture}
    \qquad
    \begin{tikzpicture}[scale=6]
    \draw[white] (0,-0.19) -- (0.01,-0.19);
        \fill[green!30!white] (0.1,0) -- (0.1,1) -- (1,1) -- (1,0) -- (0.1,0);
        \foreach \i in {0,...,9}{
        \fill[pattern= horizontal lines] (0,\i*0.1) -- (0.1,\i*0.1 + 0.1) -- (0.1, \i * 0.1);
        \fill[blue!30!white] (0,\i*0.1) -- (0,\i*0.1 + 0.1) -- (0.1, \i * 0.1 + 0.1);
        \draw (0,\i*0.1) -- (0,\i*0.1 + 0.1);
        \draw (0,\i*0.1 + 0.1) -- (0.1,\i*0.1 + 0.1);
        }
        \foreach \i in {0,...,8}{
                \fill[blue] (0.1, \i*0.1 + 0.1) circle(0.4pt);
        }
        \draw (0,0) -- (1,0) -- (1,1) -- (0,1) -- (0,0);
        \draw (0.1,0) -- (0.1,1);
         \draw [decorate,decoration={brace,amplitude=6pt,mirror}]
  (0,0) -- (0.1,0) node[midway, anchor=north, yshift = -0.3cm]{$\frac{\sigma}{2\theta}$};
    \end{tikzpicture}
\caption{Sketch of the construction in the upper bound. Left: Sketch of $\nabla u$ which acts as a building block for $\tilde{\beta}$. The striped area is where $\nabla u$ is nonconstant and behaves roughly as $\frac{\sigma}{|(x,y) - (\frac{\sigma}{2\theta}, \frac{\sigma}{2\theta})}$. Right: Sketch of the function $\tilde{\beta}$. In the green and blue region $\tilde{\beta}$ is constantly $(1,1)^T$ and $(1,1-2\theta)^T$, respectively. The blue dots indicate the support of the measure $\operatorname{curl } \tilde{\beta} = \sigma \sum_{k=1}^{\lfloor \frac{\sigma}{2\theta} \rfloor}\delta_{(\frac{\sigma}{2\theta}, k \frac{\sigma}{2\theta})^T}$. }
\label{fig: upper bound}
\end{figure}
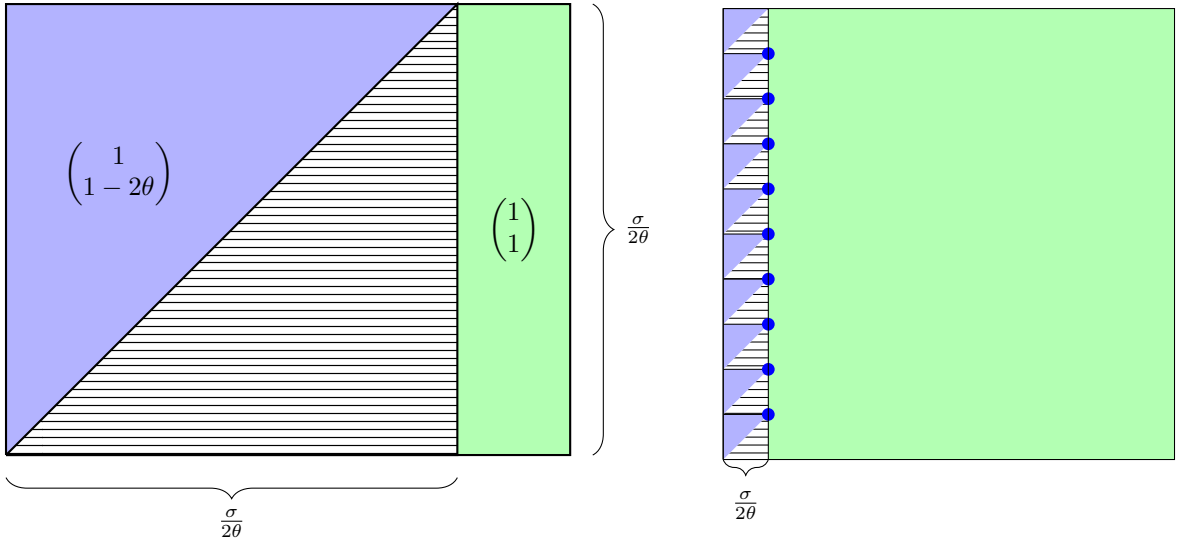    
    Then $u \in W^{1,1}((-1,2) \times (0,\frac{\sigma}{2\theta}))$ and
    \begin{equation}\label{eq: nabla u}
    \nabla u(x,y) = \begin{cases}
        \begin{pmatrix} 1 \\ 1-2\theta \end{pmatrix} &\text{ if } y \geq x, \\
    \begin{pmatrix} 1 + 2\theta \frac{ y}{\frac{\sigma}{2\theta} - y} \\ 1-2\theta + 2\theta \frac{x-2y}{\frac{\sigma}{2\theta} - y} + 2 \theta \frac{x-y}{\left(\frac{\sigma}{2\theta} - y\right)^2}y   \end{pmatrix}  &\text{ if } y<x<\frac{\sigma}{2\theta}, \\
    \begin{pmatrix} 1 \\ 1 \end{pmatrix} &\text{ if } x \geq \frac{\sigma}{2\theta}.
    \end{cases}
    \end{equation}
    Next, note that for $0 < y < x < \frac{\sigma}{2\theta}$ it holds that $\frac{\sigma}{2\theta} - x < \frac{\sigma}{2\theta} - y$, $x-y < \frac{\sigma}{2\theta} - y$ and $2\theta \leq \frac{\sigma}{\frac{\sigma}{2\theta} - y}$. 
    Consequently, it follows for $0 < y < x < \frac{\sigma}{2\theta}$ that
    \begin{align}
    \left|\nabla u(x,y) - \begin{pmatrix}
        1 \\ 1
    \end{pmatrix} \right| & \leq 2\theta +  2\theta \frac{x + 3y}{\frac{\sigma}{2\theta}-y} + 2\theta \frac{x-y}{\left(\frac{\sigma}{2\theta}-y\right)^2}y \\
    &\leq 6 \frac{\sigma}{\frac{\sigma}{2\theta} - y} \leq 12 \frac{\sigma}{\left(\frac{\sigma}{2\theta} - y\right) + \left(\frac{\sigma}{2\theta} - x\right)} \leq 12 \frac{\sigma}{|\left(\frac{\sigma}{2\theta},\frac{\sigma}{2\theta}\right)^T - (x,y)^T|}. \label{eq: diff nabla u}
    \end{align}
    Similarly, one shows that $\nabla u \in BV(\Omega \setminus B_{r}((\frac{\sigma}{2\theta},\frac{\sigma}{2\theta}));\R^2)$ for all $r>0$ and 
    \begin{equation}\label{eq: est Dnabla u}
        |D\nabla u| \leq \frac{C\sigma}{|(x,y)^T - (\frac{\sigma}{2\theta},\frac{\sigma}{2\theta})^T|^2} \chi_{\{x \leq \frac{\sigma}{2\theta}\}} \mathcal{L}^2 + \frac{C\sigma}{|y - \frac{\sigma}{2\theta}|} \mathcal{H}^1_{|\{\frac{\sigma}{2\theta}\} \times (0,\frac{\sigma}{2\theta}) \cup \{x = y\}}.
    \end{equation}
    Let us now define $v: \R^2 \to \R$ and $\tilde{\beta} : \R^2 \to \R^2$ as (cf.~Figure \ref{fig:constr v})
\begin{figure}[t]
\begin{tikzpicture}[scale = 5]
    \draw[->] (0,0) -- (1.1,0);
    \draw[yellow!70!black,thick] (0,0) -- (1,1/2);
    \foreach \i in {0,1,2,3,4}{
    \draw[red!50!white] (\i*0.2,\i*0.2*1/2) -- (\i*0.2+0.2,\i*0.2*1/2+0.2); 
    }
    \draw (1,1/2) node[anchor=north, yshift=-0.3cm, color=yellow!70!black] {$(1-2\theta)y$};
    \draw (0.2,0.2) node[anchor=south,yshift=0.2cm, color=red!50!white] {$v(\frac{\sigma}{2\theta},\cdot)$};
    \draw [decorate,decoration={brace,amplitude=6pt,mirror}]
  (0,0) -- (0.2,0) node[midway, anchor=north, yshift = -0.3cm]{$\frac{\sigma}{2\theta}$};
    \draw [decorate,decoration={brace,amplitude=6pt,mirror}]
  (1,1/2) -- (1,1/2+0.1) node[midway, anchor=north, xshift = 0.4cm, yshift = 0.2cm]{$\sigma$};
\end{tikzpicture}
    \caption{Sketch of $v(\frac{\sigma}{2\theta},\cdot)$ which only uses slope $1$ and approximates a function with slope $(1-2\theta)$ through jumps of distance $\frac{\sigma}{2\theta}$ with jump height $\sigma$. In the absolutely continuous part of $Dv$, denoted by $\tilde{\beta}$ in the proof of the upper bound, this induces the condition $\operatorname{curl } \tilde{\beta} = \sum_{k \in \mathbb{Z}} \sigma \, \delta_{(\frac{\sigma}{2\theta},k\frac{\sigma}{2\theta})}$.  }
    \label{fig:constr v}
\end{figure}
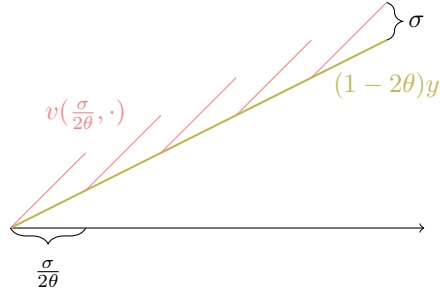
    
    \[
    v(x,y) = u(x, y - \lfloor \frac{2\theta y}{\sigma} \rfloor \frac{\sigma}{2\theta}) + (1-2\theta) \lfloor \frac{2\theta y}{\sigma} \rfloor \frac{\sigma}{2\theta} \text{ and } \tilde{\beta}(x,y) = \nabla u(x,y - \lfloor \frac{2\theta y}{\sigma} \rfloor \frac{\sigma}{2\theta}).
    \] 
    Then $v \in SBV_{loc}(\R^2)$ and $\tilde{\beta} \in L^1_{loc}(\R^2)$ with 
    \[
    J_v = \bigcup_{k \in \mathbb{Z}} (\frac{\sigma}{2\theta},\infty) \times \{ k \frac{\sigma}{2\theta} \}, \, (D v)_{| 
    J_v} = - \sigma e_2 \mathcal{H}^1_{| J_v} \text{ and } (D v)_{| \R^2 \setminus J_v} = \tilde{\beta} \, \mathcal{L}^2.
    \]
    In particular, it holds 
    \[
    \tilde{\beta}_2 = 1-2\theta \text{ in } (-\infty,0) \times \R, \text{ and } \operatorname{curl} \tilde{\beta} = \sum_{k \in \mathbb{Z}} \sigma \, \delta_{(\frac{\sigma}{2\theta},k\frac{\sigma}{2\theta})}.
    \]
    Moreover, it follows from the form of $\nabla u$, \eqref{eq: nabla u}, and \eqref{eq: diff nabla u}
    \begin{align}
    \left|\tilde{\beta}(x,y) - \begin{pmatrix} 1 \\ 1 \end{pmatrix}\right| &\leq C \left( \frac{\sigma}{\operatorname{dist}((x,y),\bigcup_{k \in \mathbb{Z}} \{ (\frac{\sigma}{2\theta},k\frac{\sigma}{2\theta}) \}} + 2\theta \right) \chi_{\{x \leq \frac{\sigma}{2\theta}\}} \\
    &\leq C \frac{\sigma}{\operatorname{dist}((x,y),\bigcup_{k \in \mathbb{Z}} \{ (\frac{\sigma}{2\theta},k\frac{\sigma}{2\theta}) \}}  \chi_{\{x \leq \frac{\sigma}{2\theta}\}}. \label{eq: est beta}  
    \end{align}
    Then we define $\beta: \R^2 \to \R^2$ as $\beta(x,y) = (\tilde{\beta}*\rho_{\eps})\left(x-\frac{\sigma}{2\theta},y\right)$.
    It follows directly that
    \begin{equation}\label{eq: competitor admissible}
    \beta_2(0,\cdot) = 1-2\theta \text{ and } \operatorname{curl } \beta = \sum_{k \in \mathbb{Z}} \sigma \, \delta_{(\frac{\sigma}{\theta},k\frac{\sigma}{2\theta})}*\rho_{\eps}.
    \end{equation}
    
    In addition, one can show using the bound \eqref{eq: est beta} and the properties of convolutions that
    \[
    \operatorname{dist}\left(\beta(x,y), K \right) \leq C  \min\left\{ \frac{\sigma}{\eps}, \frac{\sigma}{\operatorname{dist}((x,y),\bigcup_{k \in \mathbb{Z}} \{ (\frac{\sigma}{\theta},k\frac{\sigma}{2\theta})\}} \right\}  \chi_{\{ x \leq 3 \frac{\sigma}{2\theta}\}}. 
    \]
    Then we estimate using Lemma \ref{lemma: est W}
    \begin{align}
    \int_{\Omega} W(\beta) \, d\mathcal{L}^2 
    & \leq C \frac{\theta}{\sigma} \left( \int_{B_\eps(0)} \frac{\sigma^4}{\eps^4} + \frac{\sigma^2}{\eps^2} \, d \mathcal{L}^2 + \int_{B_{\sigma}(0) \setminus B_{\eps}(0)} \frac{\sigma^4}{|z|^4} \, dz + \int_{B_{\frac{3\sigma}{2\theta}(0)}\setminus B_{\sigma}(0)} \frac{\sigma^2}{|z|^2} \right) \\
    &\leq C \frac{\theta}{\sigma} \left( \frac{\sigma^4}{\eps^2} + \sigma^2 + \sigma^4 \int_{\eps}^{\infty} r^{-3} \, dr   + \sigma^2 \int_{\sigma}^{\frac{3\sigma}{2\theta}} r^{-1} \, dr\right) \\
    &\leq C\left( \theta \frac{\sigma^{3}}{\eps^2} + \theta \sigma \left( 1 + \log \frac{3}{2\theta}  \right) \right) \\
    &\leq C\left( \theta \frac{\sigma^3}{\eps^2} + \theta \sigma |\log \theta| \right), \label{eq: est int W}
    \end{align}
    where we used for the last estimate that $0 < \theta \leq 1/2$.
    
It remains to estimate $|D\beta|(\Omega)$ and $\int_{\Omega} |D\beta|^2 \, d\mathcal{L}^2$.
First, we obtain from \eqref{eq: est Dnabla u} that $\tilde{\beta} \in BV_{loc}(\R^2 \setminus \bigcup_{k \in \Z} \{ (\frac{\sigma}{2\theta}, k \frac{\sigma}{2\theta}) \};\R^2)$ with
\begin{align}
|D\tilde{\beta}| \leq & \frac{C\sigma}{\operatorname{dist}((x,y),\bigcup_{k \in \mathbb{Z}} \{ (\frac{\sigma}{2\theta},k\frac{\sigma}{2\theta}) \})^2}  \chi_{\{x \leq \frac{\sigma}{2\theta}\}} \mathcal{L}^2 + \frac{C \sigma}{\operatorname{dist}((x,y), \bigcup_{k \in \Z} \{(\frac{\sigma}{2\theta},k \frac{\sigma}{2\theta})\})} \mathcal{H}^1_{|L} \\& + 2\theta \mathcal{H}^1_{| \bigcup_{k\in \Z} (0,\frac{\sigma}{2\theta}) \times \{ k \frac{\sigma}{2\theta}\}},
\end{align}
where we set
\[
L = \left( \{\frac{\sigma}{2\theta}\} \times \R \right) \cup \bigcup_{k \in \Z} \left\{ (t, k \frac{\sigma}{2\theta} + t): t \in (0,\frac{\sigma}{2\theta}  \right\}.
\]
It follows for $(x,y) \in \R^2 \setminus \bigcup_{k \in \Z} B_{2 \eps}((\frac{\sigma}{2\theta}, k\frac{\sigma}{2\theta}))$ that
\begin{align}
    &|D \beta (x,y)| \\ = &|D\tilde{\beta} * \rho_{\eps}(x - \frac{\sigma}{2\theta},y)| \\ \leq &\frac{C\sigma}{\operatorname{dist}((x,y),\bigcup_{k \in \mathbb{Z}} \{ (\frac{\sigma}{2\theta},k\frac{\sigma}{2\theta}) \})^2}  \chi_{\{x \leq 3\frac{\sigma}{\theta}\}} + \frac{C\sigma}{\eps \operatorname{dist}((x,y), \bigcup_{k \in \Z} \{(\frac{\sigma}{\theta}, k \frac{\sigma}{2\theta})\})} \chi_{B_{\eps}(L + \frac{\sigma}{2\theta} e_1) } \\ & \quad  + \frac{C \theta}{\eps} \chi_{\bigcup_{k\in \Z} (0,\frac{\sigma}{\theta}) \times (k \frac{\sigma}{2\theta}-\eps, k \frac{\sigma}{2\theta}+\eps)}
\end{align}
Hence, we may estimate
\begin{align}
    &\int_{\Omega \setminus  \bigcup_{k \in \Z} B_{2 \eps}((\frac{\sigma}{2\theta}, k\frac{\sigma}{2\theta}))} |D\beta| \, dx \\
    \leq & C \frac{\theta}{\sigma} \left( \int_{B_{3\frac{\sigma}{\theta}}(0) \setminus B_{\eps}(0)} \frac{\sigma}{|(x,y)|^2} \, d\mathcal{L}^2 + \int_{B_{\eps}(L) \setminus B_{2 \eps}((\frac{\sigma}{2\theta}, k\frac{\sigma}{2\theta}))}  \frac{\sigma}{\eps \operatorname{dist}((x,y), \bigcup_{k \in \Z} \{(\frac{\sigma}{\theta}, k \frac{\sigma}{2\theta})\})} \, d\mathcal{L}^2 \right. \\ & \left. \qquad + \int_{0}^{\sigma/\theta} \int_{-\eps}^{\eps} \frac{\theta}{\eps} \, dy dx \right) \\
    = &C \frac{\theta}{\sigma} \left( \sigma \log\left( \frac{3 \sigma}{\theta \eps} \right) + \sigma  \right) \\
    \leq &C \left( \theta |\log(\theta)| + \theta \frac{\sigma^2}{\eps^2} \right), \label{eq: est L1 outside}
\end{align}
where we used that $\log\left( \frac{3 \sigma}{\theta \eps} \right) = |\log(\theta)| + \log\left(3\frac{\sigma}{\eps}\right) \leq |\log(\theta)| + 9\frac{\sigma^2}{\eps^2}$.
Similarly, we estimate
\begin{align}
    &\int_{\Omega \setminus  \bigcup_{k \in \Z} B_{2 \eps}((\frac{\sigma}{2\theta}, k\frac{\sigma}{2\theta}))} |D\beta|^2 \, d\mathcal{L}^2 \\
    \leq & C \frac{\theta}{\sigma} \left( \int_{B_{3\frac{\sigma}{\theta}}(0) \setminus B_{\eps}(0)} \frac{\sigma^2}{|(x,y)|^4} \, d\mathcal{L}^2 + \int_{B_{\eps}(L) \setminus B_{2 \eps}((\frac{\sigma}{2\theta}, k\frac{\sigma}{2\theta}))}  \frac{\sigma^2}{\eps^2 \operatorname{dist}((x,y), \bigcup_{k \in \Z} \{(\frac{\sigma}{\theta}, k \frac{\sigma}{2\theta})\})^2} \, d\mathcal{L}^2 \right. \\ & \qquad \left.+ \int_{0}^{\sigma/\theta} \int_{-\eps}^{\eps} \frac{\theta^2}{\eps^2} \, dy dx \right) \\
    = & C \frac{\theta}{\sigma} \left( \frac{\sigma^2}{\eps^2}  + \frac{\sigma \theta}{\eps}  \right) 
    \leq C  \theta \frac{\sigma}{\eps^2}, \label{eq: est L2 outside}
\end{align}
where we used that $0 < \eps < \sigma$ implies that $\frac{\theta^2}{\eps} \leq \theta \frac{\sigma}{\eps^2}$.
Next, we estimate the contributions on $\bigcup_{k \in \Z} B_{2 \eps}((\frac{\sigma}{\theta}, k\frac{\sigma}{2\theta}))$.
Using Young's inequality for convolutions we find using \eqref{eq: est beta}
\begin{align}
    \int_{\Omega \cap \bigcup_{k \in \Z} B_{2 \eps}((\frac{\sigma}{\theta}, k\frac{\sigma}{2\theta}))} |D\beta| \, d\mathcal{L}^2  \leq C \frac{\theta}{\sigma} \int_{B_{2\eps}((\frac{\sigma}{\theta}, \frac{\sigma}{2\theta})} |D\beta| \, d\mathcal{L}^2
    \leq &C \frac{\theta}{\sigma}\| \tilde{\beta} \|_{L^1(B_{3\eps}((\frac{\sigma}{2\theta}, \frac{\sigma}{2\theta})))} \| D \rho_{\eps} \|_{L^1(B_{\eps}(0))} \\ 
    \leq & C \frac{\theta}{\sigma} \left(\int_{0}^{3\eps} r \frac{\sigma}{r} \, dr\right) \frac1{\eps} \leq C \theta. \label{eq: est L1 inside}
\end{align}
Similarly, we obtain
\begin{align}
    \int_{\Omega \cap \bigcup_{k \in \Z} B_{2 \eps}((\frac{\sigma}{\theta}, k\frac{\sigma}{2\theta}))} |D\beta|^2 \, d\mathcal{L}^2 \leq C \frac{\theta}{\sigma} \int_{B_{2\eps}((\frac{\sigma}{\theta}, \frac{\sigma}{2\theta})} |D\beta|^2 \, d\mathcal{L}^2 
    \leq &C \frac{\theta}{\sigma}\| \tilde{\beta} \|_{L^1(B_{3\eps}((\frac{\sigma}{2\theta}, \frac{\sigma}{2\theta})))}^2 \| D \rho_{\eps} \|_{L^2(B_{\eps}(0))}^2 \\ 
    \leq & C \frac{\theta}{\sigma} \sigma \eps \frac1{\eps^2} = C  \theta \frac{\sigma}{\eps^2}. \label{eq: est L2 inside}
\end{align}
In particular, it follows together with \eqref{eq: competitor admissible} that $\beta \in \mathcal{A}_{\sigma,\theta,\eps}^{(\epsilon)}$ for $\epsilon \in \{1,2,a\}$.

Combining \eqref{eq: est int W}, \eqref{eq: est L1 outside} and \eqref{eq: est L1 inside} yields
\[
E^{(1)}_{\sigma,\theta,\eps}(\beta) = \int_{\Omega} W(\beta) \, d\mathcal{L}^2 + \sigma |D\beta|(\Omega) \leq C \left( \theta \frac{\sigma^3}{\eps^2} + \theta \sigma |\log \theta| \right),
\]
whereas \eqref{eq: est int W}, \eqref{eq: est L2 outside} and \eqref{eq: est L2 inside} imply
\[
E^{(a)}_{\sigma,\theta,\eps}(\beta) \leq E^{(2)}_{\sigma,\theta,\eps}(\beta) = \int_{\Omega} W(\beta) + \sigma^2 |D\beta|^2 \, d\mathcal{L}^2 \leq C \left( \theta \frac{\sigma^3}{\eps^2} + \theta \sigma |\log \theta| \right).
\]

\end{proof}

\section{Lower bound}\label{sec: lower}
In this section we prove the lower bound of Theorem \ref{thm: main}. 

\subsection{Energy estimates for vortices and the ball construction}

We start with a lemma that will provide a lower bound on the energy concentrated within the $\eps$-ball around vortices.

\begin{lemma}\label{lem: lb vortex general}
    There exists $c_v>0$ such that for all $\sigma \geq \sqrt{2} \pi \eps >0$ and $\beta: B_{\eps}(0) \to \R^2$ such that $\operatorname{curl} \beta = \sigma \delta_0 * \rho_{\eps}$ it holds
    \[
    \int_{B_{\eps}(0)} W(\beta) \, d\mathcal{L}^2 \geq c_v \frac{\sigma^4}{\eps^2}.
    \]
\end{lemma}
\begin{proof}
        First, note using Young's inequality that 
\begin{align*}
W(\beta) &= (1-\beta_1^2)^2 + (1-\beta_2^2)^2 = \beta_1^4 + \beta_2^4 -2 (\beta_1^2 + \beta_2^2) + 2\geq \frac12 |\beta|^4 - 2 |\beta|^2 + 2 \geq \frac14 |\beta|^4 -2.
\end{align*}
Now, estimate
\begin{align}
    4\int_{B_{\eps}(0)}  W(\beta) \, d\mathcal{L}^2 + 8 \pi \eps^2 &\geq \int_{B_{\eps}(0)} 4 W(\beta) \, d\mathcal{L}^2 + 8 \pi \eps^2  \\ &\geq \int_{B_{\eps}(0)} |\beta|^4 \, d\mathcal{L}^2 \\& \geq \int_0^{\eps} \int_{\partial B_r(0)} |\beta \cdot \tau|^4 \, d\mathcal{H}^1 \, dr \\ & \geq \int_{\eps / 2}^{\eps} \frac1{(2\pi r)^3} \left| \int_{\partial B_r(0)} \beta \cdot \tau \, d\mathcal{H}^1 \right|^4 \, dr \\
    &= \frac1{8 \pi^3} \int_{\eps / 4}^{\eps} r^{-3} \left| \int_{B_r(0)}  \sigma \rho_{\eps}(x) \, d\mathcal{L}^2 \right|^4 \, dr \\
    &=  \frac{15\sigma^4}{16 \pi^3 \eps^2} \geq 15 \pi \eps^2 \geq \frac32 (8\pi \eps^2).
\end{align}
Consequently, it holds  
\[
4\int_{B_{\eps}(0)} W(\beta) \, d\mathcal{L}^2 \geq \frac{15\sigma^4}{16\pi^3 \eps^2} - 8 \pi \eps^2 \geq  \frac{5\sigma^4}{16 \pi^3 \eps^2}.
\]
\end{proof}

Next, we prove a one-dimensional lemma that will guarantee lower bounds for the energy on slices either from $W(\beta)$ or from $D\beta$. 

\begin{lemma}\label{lemma: alternative}
    For all $I \subseteq \R$  open the following holds:  
    \begin{enumerate}
        \item For all $v \in BV(I)$ one of the following is true
        \begin{enumerate}
            \item $|Dv|(I) \geq \frac12$;
            \item there exists $\xi \in \{\pm 1\}$ such that it holds for all $y\in I$ that $|v(y) - \xi| \leq 3 \min_{\xi' = \pm 1} |v(y) - \xi'|$.
        \end{enumerate}
        \item Let $\Phi: \R \to \R$ be the function from \eqref{eq: def phi}. For all $\sigma > 0$ and $v \in W^{1,2}(I)$ one of the following is true
        \begin{enumerate}
            \item $\int_I (1 - v(y)^2)^2 + \sigma^2 |v'(y)|^2 \, dy \geq 2 \sigma \int_I |(\Phi\circ v)'(y)| \geq \frac{\sigma}{16}$;
            \item there exists $\xi \in \{\pm 1\}$ such that it holds for all $y\in I$ that $|v(y) - \xi| \leq 3 \min_{\xi' = \pm 1} |v(y) - \xi'|$.
        \end{enumerate}
    \end{enumerate}
\end{lemma}
\begin{proof}
    First, let $v \in BV(I)$. We assume that $|Dv|(I) \leq 1/2$. Next, let $x \in I$ and $\xi \in \{\pm 1\}$ such that $|v(x) - \xi| = \min_{\xi' = \pm 1} |v(x) - \xi'|$. Since $|Dv|(I) \leq 1/2$, we have for all $y  \in I$ that $|v(x) - v(y)| \leq 1/2$. This implies that $|v(y) - \xi| \leq 3 \min_{\xi' = \pm 1} |v(y) - \xi'|$. \\
    Next we prove (2). The first inequality in (a) follows directly by Young's inequality. Let $v \in W^{1,2}(I)$ and assume that $\int_I |(\Phi \circ v)'| \, d\mathcal{L}^1 \leq \frac{1}{32}$. Next, let $x \in I$ and $\xi \in \{\pm 1\}$ such that $|v(x) - \xi| = \min_{\xi' = \pm 1} |v(x) - \xi'|$. Then it follows by Lemma \ref{lem: properties phi} for all $y \in I$ that $\frac18 |v(x) - v(y)|^2 \leq |\Phi(v(x)) - \Phi(v(y))| \leq \frac1{32}$, i.e.  $|v(x) - v(y)| \leq \frac12$. The conclusion follows as in (1).
\end{proof}
\begin{remark}\label{rem: alternative}
Note that it holds for $\beta \in \R^2$ that 
\[
\min_{\xi \in K} |\beta_1 - \xi_1| + |\beta_2 - \xi_2| \leq \sqrt{2} \operatorname{dist}(\beta,K).
\]
    Hence, for $v: I \to \R^2$ it is straightforward to obtain the analogous statement to Lemma \ref{lemma: alternative} where one has to replace (b) by the statement: there exists $\xi \in K$ such that it holds for almost all $y \in I$ that $|v(y) - \xi| \leq 3\sqrt{2} \operatorname{dist}(v(y),K) \leq 3\sqrt{2} \sqrt{W(v(y))}$.
\end{remark}

Next, we present an estimate on annuli that will be the fundamental ingredient in a ball construction (see Proposition \ref{prop: ball constr} below) to prove a logarithmic lower bound for the energy which is distributed in balls of radius $\frac{\sigma}{\eps \theta}$ around vortices. The main difference to similar estimates in the context of dislocation models (see \cite{dLGaPo12,gin18,GiNeZw24}) is that there is no rigidity for the term involving $W(\beta)$. In principle, the four different values in $K$ could be exploited by $\beta$ to create a circulation  of the form $\int_{\partial B_r(0)} \beta \cdot \tau \, d\mathcal{H}^1= \sigma$ (at least for large enough values of $r$). However, this would come at the expense of creating energy in the term involving $D\beta$. Hence, the bound below involves both terms of the energy.  

\begin{proposition}\label{prop: est annulus}
    There exists $c>0$ such that the following is true. Let $0 <r < R$. Then it holds for all $\beta: B_R(0) \to \R^2$ with $\operatorname{curl } \beta \in \mathcal{M}(B_R(0))$, $(\operatorname{curl} \beta) (B_r(0)) \in \sigma \Z$ and $\operatorname{curl } \beta = 0$ in $A:=B_R(0) \setminus B_r(0)$ that  
    \begin{equation}\label{eq: est annulus1}
    \int_{A} W(\beta) \, dx + \sigma |D\beta|(A) \geq c \sigma |\operatorname{curl } \beta(B_r(0))| \log(R/r)
    \end{equation}
    and
    \begin{equation}\label{eq: est annulus quadratic}
    \int_{A} W(\beta) + 2 \sigma |D(\Phi \circ \beta_1)| + 2 \sigma |D(\Phi \circ \beta_2)| \, dx \geq c \sigma |\operatorname{curl } \beta(B_r(0))| \log(R/r).
    \end{equation}
\end{proposition}

\begin{proof}
    We first prove estimate \eqref{eq: est annulus1}. Let $t \in (r,R)$. 
    First, assume that $\frac{|\operatorname{curl } \beta(B_r(0))|^2}{t^2} \geq 512 $. 
    Then we estimate using $|\beta|^2 \leq 2\operatorname{dist}\left(\beta, K \right)^2 + 4 \leq 2 W(\beta) + 4$, Jensen's inequality and $\left( \operatorname{curl } \beta \right)(B_r(0))  \in \sigma \Z$
    \begin{align}
        \int_{\partial B_t(0)} W(\beta) \, d\mathcal{H}^1       &\geq \int_{\partial B_t(0)} (\frac12 |\beta|^2 - 2) \, d\mathcal{H}^1 \\
        &\geq \frac1{16t} \left( \int_{\partial B_t(0)} \beta\cdot \tau\, d\mathcal{H}^1 \right)^2 - 16t \\
        &= \frac1{16t} |\operatorname{curl} \beta (B_r(0))|^2  - 16t \\
        &\geq \frac1{32t} |\operatorname{curl} \beta (B_r(0))|^2 \geq \frac1{32t} \sigma |\operatorname{curl} \beta (B_r(0))|. \label{eq: curl large1}
    \end{align}
Next, assume that $\frac{|\operatorname{curl } \beta(B_r(0))|^2 }{t^2} \leq 512$ which implies that $\frac{|\operatorname{curl } \beta(B_r(0))| }{t} \leq 25$.   
    Next, note that by Remark \ref{rem: alternative} one of the following is true on $\partial B_t(0)$:
    \begin{enumerate}
        \item $|\partial_{\tau} \beta|(\partial B_t(0))| \geq \frac12$, where $\partial_{\tau} \beta$ denotes the tangential derivative of $\beta$ on $\partial B_t(0)$;
        \item there exists $\xi \in K$ such that $|\beta(y) - \xi|^2 \leq 18  W(\beta(y)) $ for all $y \in \partial B_t(0)$.
    \end{enumerate}
    
    Let us now assume that (2) holds. Then we estimate again using Jensen's inequality analogously to above
    \begin{align*}
    18 \int_{\partial B_t(0)} W(\beta) \, d \mathcal{H}^1 &\geq \int_{\partial B_t(0)} |\beta - \xi|^2 \, d\mathcal{H}^1 \\
    &\geq \frac1{8 t} \left( \int_{\partial B_t(0)} (\beta-\xi)\cdot \tau  \, \mathcal{H}^1 \right)^2 \\
    &= \frac{1}{8t} \left| \operatorname{curl} \beta (B_{r}(0)) \right|^2 \geq \frac{1}{8t} \sigma\left| \operatorname{curl} \beta (B_{r}(0)) \right|.
    \end{align*}
    Then, we estimate using (1) or (2) and $\frac{|\operatorname{curl }\beta (B_r(0))|}{144 t} \leq \frac{25}{144} \leq \frac12$
    \begin{align}
        \int_{\partial B_t(0)} W(\beta) \, d\mathcal{H}^1 + \sigma |D\beta|(\partial B_t(0)) &\geq \min\left\{ \int_{\partial B_t(0)} W(\beta) \, d \mathcal{H}^1, \sigma |\partial_{\tau} \beta|(\partial B_t(0)) \right\}  \\
        &\geq \min\left\{ \frac{1}{144t} \sigma\left| \operatorname{curl} \beta (B_r(0)) \right|, \sigma/2 \right\}  \\
        &= \frac{1}{144t} \sigma |\operatorname{curl} \beta (B_r(0))|. \label{eq: curl small 1}
    \end{align}
    Combining \eqref{eq: curl large1} and \eqref{eq: curl small 1} then yields
    \begin{align}
    \int_{B_R(0) \setminus B_r(0)} W(\beta) \, d\mathcal{L}^2 + \sigma |D\beta|(A) &\geq \int_{r}^R \left(\int_{\partial B_t(0)} W(\beta) \, d\mathcal{H}^1 + \sigma |D\beta|(\partial B_t(0)) \right) \\ &\geq \int_r^R \frac1{144t} \sigma |\operatorname{curl} \beta(B_r(0))| \, dt = \frac1{144} \sigma |\operatorname{curl} \beta(B_r(0))| \log(R/r).
    \end{align}

    Estimate \eqref{eq: est annulus quadratic} follows analogously using the second statement of Lemma \ref{lemma: alternative} and Remark \ref{rem: alternative}, respectively. 

\end{proof}

Next, we will briefly recall the celebrated ball construction that was developed in the seminal papers \cite{Je99,Sa98} in the context of vortices in the Ginzburg-Landau model (see also \cite{gin18,GiNeZw24,dLGaPo12} for applications in the context of dislocation model and \cite{GiGl25} to surgically remove the jump set of an $SBV^p$-funtion).

\begin{lemma}[Ball-construction, \cite{GiNeZw24}]\label{lemma: ball construction}
    Let $(B_{r_i}(p_i))_{i \in I} $ be a finite family of open balls in $\R^2$. Then for every $t>0$ there exists a finite family of open balls $(B_{r_i(t)}(p_i(t)))_{i \in I(t)}$ with pairwise disjoint closures such that the following properties hold: 
    \begin{enumerate}
        \item $\sum_{i \in I(t)} r_i(t) \leq e^t \sum_{i \in I} r_i$,
        \item $\bigcup_{i \in I} B_{r_i}(p_i) \subseteq \bigcup_{i \in I(t)} B_{r_i(t)}(p_i(t))$,
        \item for all $s \in [0, t]$ and $i \in I(s)$ there exists  a unique $j \in I(t)$ such that $
B_{e^{t-s}r_i(s)}(p_i(s)) \subseteq B_{r_j(t)}(p_j(t))$,
\item there exists $0<t_1 < \dots < t_N$ with $N \leq \#I$ such that for every $t_n < t < t_{n+1}$ it holds $I(t) = I(t_n)$ and for every $i \in I(t)$ it holds $p_i(t) = p_i(t_n)$ and $r_i(t) = e^{t-t_n}r_i(t_n)$.
    \end{enumerate}
\end{lemma}

Combining this construction with the lower bound on annuli, Proposition \ref{prop: est annulus}, yields the following logarithmic lower bound on the energy.

\begin{proposition}\label{prop: ball constr}
    There exists $c_{bc} > 0$ such that the following holds: Let $\sigma > 0$, $A' \subseteq A\subseteq \R^2$ with $\operatorname{dist}(A',\partial A) =: d$ and $\beta: A \to \R^2$ such that $\operatorname{curl} \beta = \sigma \sum_{i=1}^n \gamma_i \delta_{x_i} * \rho_{\eps}$, $\gamma_i \in \{\pm 1\}$ and $x_i \in \R^2$. Let $(B_{r_i(t)}(p_i(t)))_{i \in I(t)}$ be the output of the ball construction from Lemma \ref{lemma: ball construction} starting with the balls $B_{\eps}(x_i)$. Then it holds for all $T>0$ such that $\eps n e^T \leq d/2$ and $J(T) = \{ i\in I(T): B_{r_i(T)}(p_i(T)) \cap A' \neq \emptyset \}$ 
    \begin{equation}
    \int_{A} W(\beta) \, d\mathcal{L}^2 + \sigma |D\beta|(A) \geq c_{bc} \sigma T \, \left|\operatorname{curl } \beta(\bigcup_{ i \in J(T)} B_{r_i(T)}(p_i(T)))\right|
    \end{equation}
    and
    \begin{equation}
    \int_{A} W(\beta) + 2 \sigma |D(\Phi \circ \beta_1)| + 2 \sigma |D(\Phi \circ \beta_2)| \, d\mathcal{L}^2 \geq c_{bc} \sigma T \, \left|\operatorname{curl } \beta(\bigcup_{i \in J(T)} B_{r_i(T)}(p_i(T)))\right|.
    \end{equation}
\end{proposition}
\begin{proof}
We only prove the first inequality. The second inequality can be shown completely analogously.
Let $0< t_1 < \dots < t_{N} < T$, $N \leq n$, be the merging time of the ball construction up to time $T$, see property (4) in Lemma \ref{lemma: ball construction}. We set $t_{N+1} = T$ and $t_0 = 0$. 
Now, fix $i \in I(t_m)$, $0 \leq m \leq N$. Then it follows by Proposition \ref{prop: est annulus} that 
\begin{align}
&\int_{B_{e^{t_{m+1} - t_m}r_i(t_{m})}(p_i(t_{m})) \setminus B_{r_i(t_{m})}(p_i(t_{m}))} W(\beta) \, d\mathcal{L}^2 + \sigma |D\beta|(B_{e^{t_{m+1} - t_m}r_i(t_{m})}(p_i(t_{m})) \setminus B_{r_i(t_{m})}(p_i(t_{m}))) \\ \geq &c \sigma |(\operatorname{curl } \beta)( B_{r_i(t_{m})}(p_i(t_{m})))| \, (t_{m+1} - t_m).
\end{align}
Eventually, we notice that by property (1) of Lemma \ref{lemma: ball construction} it holds for all $B_{r_i(T)}(p_i(T)) \cap A' \neq \emptyset$ that $r_i(T) \leq n \eps e^T \leq d/2$ and therefore $B_{r_i(T)}(p_i(T)) \subseteq A$. 
For $i \in I(T)$ let us additionally define 
\[
J_i(t_m) = \{ j \in I(t_m): B_{r_j(t_m)}(p_j(t_m)) \subseteq B_{r_i(T)}(p_i(T))\}.
\]
Summing over all $i \in J(T)$, all $j \in J_i(t_m)$ and $m$ it follows using (3)  of Lemma \ref{lemma: ball construction} 
\begin{align}
    &\int_{A} W(\beta) \, d\mathcal{L}^2 + \sigma |D\beta|(A) \\ \geq &c \sigma \sum_{i \in J(T)} \sum_{m=0}^{N} \sum_{j \in J_i(t_m)}  |\operatorname{curl } \beta(B_{e^{t_{m+1} - t_m}r_j(t_m)}(p(t_m)) \setminus B_{r_j(t_m)}(p(t_m)))| \, (t_{m+1} - t_m) \\
    \geq &c T \sigma   \left|\operatorname{curl } \beta(\bigcup_{B_{r_i(T)}(p_i(T)) \cap A' \neq \emptyset} B_{r_i(T)}(p_i(T)))\right|.
\end{align}
Here, we used that the balls $B_{r_j(t_m)}(p(t_m))$ are pairwise disjoint.
\end{proof}

\subsection{The lower bound for the energies \texorpdfstring{$E^{(1)}_{\sigma,\theta,\eps}$}{E(1) sigma,theta,epsilon} and \texorpdfstring{$E^{(2)}_{\sigma,\theta,\eps}$}{E(2) sigma,theta,epsilon}}

In this section we will prove the lower bound from Theorem \ref{thm: main} for the energies $E^{(1)}_{\sigma,\theta,\eps}$ and $E^{(2)}_{\sigma,\theta,\eps}$.
Let us briefly comment on the strategy. 
First, we will show only that $E_{\sigma,\theta,\eps}^{(\epsilon)} \gtrsim \min\{ \theta^2, \sigma \left( \frac{|\log \sigma|}{|\log \theta|} + 1 \right), \theta \frac{\sigma^3}{\eps^2} \}$. 
Then we will use Proposition \ref{prop: ball constr} to obtain the additional logarithmic term $\theta \sigma |\log \theta|$ in the last term of the energy estimate.
For technical reasons, the proof of the above estimate will be obtained separately in different parameter regimes: if $\theta \gg 0$ and $\sigma > 0$ then $\frac{|\log \sigma|}{|\log \theta|} \sim |\log \sigma|$ and it is enough to show that $E_{\sigma,\theta,\eps}^{(\epsilon)} \gtrsim \min\{ \theta^2, \sigma \left( |\log \sigma| + 1 \right), \theta \frac{\sigma^3}{\eps^2} \}$, see Proposition \ref{prop: theta large}; if $\sigma \gg \theta^{k_0}$ for some $k_0 > 0$, then $\frac{|\log \sigma|}{|\log \theta|} \lesssim k_0$ and it essentially suffices to prove $E_{\sigma,\theta,\eps}^{(\epsilon)} \gtrsim \min\{ \theta^2, \sigma, \theta \frac{\sigma^3}{\eps^2} \}$, see Proposition \ref{prop: sigma simple}; eventually we prove the full estimate only if $\sigma \leq \theta^{k_0}$ and $\theta \ll 1/2$, see Proposition \ref{eq: sigma small}.

We start with the estimate for large values of $\theta$. 

    \begin{proposition}\label{prop: theta large}
         For all $\theta_0 > 0$ there exists $c_{LB}^{(1)}>0$ (depending on $\theta_0$) such that it holds for all $\sigma \geq \sqrt{2} \pi\eps> 0$, $\theta_0 \leq \theta  \leq 1/2$ and $\epsilon \in \{1,2\}$ that
    \[
    \inf_{\beta \in \mathcal{A}^{(\epsilon)}_{\sigma,\theta,\eps}} E^{(\epsilon)}_{\sigma,\theta,\eps}(\beta) \geq c_{LB}^{(1)} \min\left\{ \theta^2, \sigma ( |\log \sigma| + 1), \theta \frac{\sigma^3}{\eps^2}   \right\}.
    \]
    \end{proposition}
    \begin{proof}
    Let $\beta \in \mathcal{A}^{(\epsilon)}_{\sigma,\theta,\eps}$.
First, let $\sigma \geq 1/4$. If $|\operatorname{curl} \beta|((0,1)^2) \neq 0$ it follows from Lemma \ref{lem: lb vortex general} that 
\[
E_{\sigma,\theta,\eps}^{(\epsilon)}(\beta) \geq c_v \frac{\sigma^4}{\eps^2} \geq \frac{c_v}{16} \geq \frac{c_v}{16} \min\left\{ \theta^2, \sigma ( |\log \sigma| + 1), \theta\frac{\sigma^3}{\eps^2}   \right\}. 
\]
On the other hand, if $\operatorname{curl } \beta = 0$ then it follows from \cite[Lemma 4]{GinZwi:22} that
\[
E_{\sigma,\theta,\eps}^{(\epsilon)}(\beta) \geq c \min\left\{ \theta^2, \sigma (|\log \sigma| + 1)  \right\}.
\]
Hence, we will now assume that $0 < \sigma \leq 1/4$. 
Define $c_1 := \min\left\{ c_v\frac{\theta_0}{32}, \frac{\theta_0}{20} \right\}$ and $c_2 := \frac{\theta_0}{20}$, where $c_v >0$ is the constant from Lemma \ref{lem: lb vortex general}.
We may assume that it holds $E_{\sigma,\theta,\eps}^{(\epsilon)}(\beta) \leq c_1 \min\left\{ \theta^2, \sigma ( |\log \sigma| + 1), \theta\frac{\sigma^3}{\eps^2}   \right\}$.
By Lemma \ref{lem: lb vortex general} this implies that 
\[
 \#\text{vortices in }(0,1)^2 \cdot c_v \frac{\sigma^4}{\eps^2} \leq E_{\sigma, \theta, \eps}^{(\epsilon)}(\beta) \leq c_1 \theta \frac{\sigma^3}{\eps^2},
\]
which implies
\[
\int_{(0,1)^2} |\operatorname{curl } \beta| \, dx \leq \theta \frac{c_1}{c_v} .
\]

Now, consider 
\[
x \in \left( c_2 \sigma^{1/2}, c_2 \right).
\]
Then find an interval $I = (y_1,y_2) \subseteq (0,1)$ with $|I| = \frac1{c_2} x$ such that
\begin{enumerate}
    \item $E_{\sigma, \theta, \eps}^{(\epsilon)}(\beta; \{x \} \times I) \leq 16 |I| E_{\sigma, \theta, \eps}^{(\epsilon)}(\beta; \{x\} \times (0,1))$,
    \item $E_{\sigma, \theta, \eps}^{(\epsilon)}(\beta; (0,1) \times \{y_i\}) \leq 16 E_{\sigma,\theta,\eps}^{(\epsilon)}(\beta)$, $i=1,2$,
    \item $\int_{(0,1) \times I} |\operatorname{curl } \beta| \, dx \leq 16 |I| \frac{c_1}{c}$.
\end{enumerate}

Next, note that since $0 < c_1 \leq c_v\frac{\theta_0}{32}$ it follows that 
\begin{equation}\label{eq: curl small}
\int_{(0,x) \times I} |\operatorname{curl } \beta| \, d\mathcal{L}^2 \leq \int_{(0,1) \times I} |\operatorname{curl } \beta| \, d\mathcal{L}^2 \leq 16 |I| \frac{c_1}{c_v} \leq \frac{\theta_0}2 |I| \leq \frac{\theta}2 |I|.
\end{equation}
Moreover, we estimate for $i=1,2$ using Lemma \ref{lemma: est W} and (2)
\begin{align}
    \int_0^x |\beta_1(s,y_i)| \, ds &\leq x + \int_0^x |\beta_1(s,y_i) \pm 1| \, ds \\
    &\leq c_2 |I| + x^{1/2} E_{\sigma,\theta,\eps}^{(\epsilon)}(\beta; (0,1) \times \{y_i\})^{1/2} \\
    &\leq c_2 |I| + 4 x^{1/2} E_{\sigma,\theta,\eps}^{(\epsilon)}(\beta)^{1/2} \\
    &\leq c_2 |I| + 4 \left( c_1 x \sigma (|\log \sigma| + 1) \right)^{1/2} \\
    &\leq 5c_2 |I| \leq \frac{\theta_0}4 |I| \leq \frac{\theta}4 |I|,
\end{align}
where we used that $c_1 \sigma (|\log \sigma| + 1) \leq c_2 \sigma^{1/2} \leq x$ since  $0 < c_1 \leq c_2$ and $c_2 = \frac{\theta_0}{20}$.
Using Lemma \ref{lemma: rewrite} we then find
\begin{align} \label{eq: int beta second}
    \left| \int_I \beta_2(x,t) \, dt  \right| &\leq  (1-2\theta) |I| + \int_0^x |\beta(s,y_2)| + |\beta(s,y_1)| \, ds + \int_{(0,x) \times I} |\operatorname{curl } \beta| \, d\mathcal{L}^2  \leq (1-\theta) |I|.
\end{align}
From now on we only consider $\epsilon = 1$. The case $\epsilon = 2$ is completely analogous.
By Lemma \ref{lemma: alternative} applied to $\beta_2(x,\cdot)$ one of the following holds on $I$:
\begin{itemize}
    \item[(4)] $|\partial_2 \beta_2(x,\cdot)|(I) \geq \frac12$,
    \item[(5)] there exists $\xi \in \{ \pm 1\}$ such that it holds  $|\beta_2(x,y) - \xi| \leq  3 \min_{\xi' = \pm 1} |\beta(x,y) - \xi'|$ for all a.e.~$y \in I$.
\end{itemize}
If (4) is true we have 
\[
  16|I| E^{(1)}_{\sigma,\theta,\eps}(\beta; \{x\} \times (0,1)) \geq E^{(1)}_{\sigma,\theta,\eps}(\beta; \{x\} \times I) \geq \sigma / 2,
\]
which implies 
\begin{equation} \label{eq: 4.}
E_{\sigma,\theta,\eps}^{(1)}(\beta; \{x\} \times (0,1)) \geq \sigma \frac{c_2}{32 x}.
\end{equation}
On the other hand, if (5) is true then it follows using Lemma \ref{lemma: est W} from \eqref{eq: int beta second}
\[
\int_I W(\beta(x,y)) \, dy \geq \frac19 \int_I |\beta(x,y) - \xi|^2 \, dy \geq \frac1{9 |I|} \left|\int_I \beta_2(x,y) - \xi \, dy \right|^2 \geq \frac{\theta^2}{9} |I|.  
\]
Using (1), we obtain
\begin{equation}\label{eq: 5.}
    E_{\sigma,\theta,\eps}^{(1)}(\beta; \{x\} \times (0,1)) \geq \frac1{16|I|} E_{\sigma,\theta,\eps}^{(1)}(\beta; \{x\} \times I) \geq \frac{\theta^2}{144}.
\end{equation}
Combining \eqref{eq: 4.} and \eqref{eq: 5.} yields using $c_2 \frac{\sigma}{x} \leq \sigma^{1/2} \leq 1$ for all $x \geq c_2 \sigma^{1/2}$
\[
E_{\sigma,\theta,\eps}^{(1)}(\beta) \geq \frac{\theta_0^2}{144}\int_{c_2 \sigma^{1/2}}^{c_2} \min\left\{ c_2\frac{\sigma}x, 1 \right\} \, dx \geq   \frac{c_2 \theta_0^2}{144} \int_{c_2 \sigma^{1/2}}^{c_2}  \frac{\sigma}x \, dx = \frac{c_2 \theta_0^2}{288} \sigma  |\log(\sigma)| \geq \frac{c_2 \theta_0^2}{576}  \sigma  (|\log(\sigma)| + 1 ).
\]
Note that we used in the last inequality that $|\log \sigma| \geq \log 4 \geq 1$.
\end{proof}

Next, we prove a lower bound that is useful as long as $\sigma \geq \theta^{k_0}$ for some $k_0 > 0$.

\begin{proposition}\label{prop: sigma simple}
    There exists $c_{LB}^{(2)}>0$ such that it holds for all $\sigma, \eps, \sigma, \theta > 0$ with $\sigma \geq \sqrt{2} \pi \eps >0$ and $\epsilon \in \{\pm 1\}$ that
    \[
    \inf_{\beta \in \mathcal{A}^{(\epsilon)}_{\sigma,\theta,\eps}} E^{(\epsilon)}_{\sigma,\theta,\eps}(\beta) \geq c_{LB}^{(2)} \min\left\{ \theta^2, \sigma, \theta \frac{\sigma^3}{\eps^2}  \right\}.
    \]
\end{proposition}

\begin{proof}
By Proposition \ref{prop: theta large} we may assume that $0 < \theta < \frac1{64}$.
Let $\beta \in \mathcal{A}^{(\epsilon)}_{\sigma,\theta,\eps}$.
Define $c_1 := \min\left\{ \frac1{576}, \frac{c_v}8 \right\}$, where $c_v$ is the universal constant from Lemma \ref{lem: lb vortex general}.
Then we may assume that
\[
 E_{\sigma,\theta,\eps}^{(\epsilon)}(\beta) \leq c_1  \min\left\{ \theta^2, \sigma, \theta \frac{\sigma^3}{\eps^2}  \right\}.
\]
Arguing as in the proof of Proposition \ref{prop: theta large} we find using Lemma \ref{lem: lb vortex general}  that
\begin{equation} \label{eq: est curl average}
 \int_{(0,1)^2} |\operatorname{curl } \beta| \, dx \leq \theta \frac{c_1}{c_v}  \leq \frac{\theta}8.
\end{equation}
Then find $x \in (3/4,1)$ and $y_1 \in (0,1/4)$ such that it holds for $y_2 = y_1 + 1/2$ that
\begin{enumerate}
    \item $E_{\sigma,\theta,\eps}^{(\epsilon)}(\beta, (0,1) \times \{y_i\}) \leq 4 E_{\sigma,\theta,\eps}^{(\epsilon)}(\beta)$ for $i=1,2$;
    \item $E_{\sigma,\theta,\eps}^{(\epsilon)}(\beta, \{x\} \times (0,1)) \leq 4 E_{\sigma,\theta,\eps}^{(\epsilon)}(\beta)$.
\end{enumerate}
By (1) we note that it holds for $i=1,2$
\begin{align}
&|\partial_1 \beta_1 (\cdot,y_i)|((0,1)) \leq \frac{4 c_1}{\sigma} E^{(1)}_{\sigma,\theta,\eps}(\beta) \leq 4 c_1 \\ \text{ and } &\int_0^1 (1-\beta_1(s,y_i)^2)^2 + \sigma^2 |\partial_1 \beta_1(s,y_i)|^2 \, ds \leq 4 E_{\sigma,\theta,\eps}^{(2)}(\beta) \leq 4c_1 \sigma.
\end{align}
Since $0 < c_1 < \frac{1}{128}$, it follows from Lemma \ref{lemma: alternative} that there exists $\xi_i \in \{\pm 1\}$ such that $|\beta_1(s,y_i)  - \xi_i| \leq 3 \min_{\xi' = \pm 1} |\beta_1(s,y_i) - \xi'|$ for all $s \in (0,1)$ and $i=1,2$. 
Using Lemma \ref{lemma: rewrite} we write
\begin{align}
&\int_0^x (\beta_1(s,y_2) - \xi_2) - (\beta_1(s,y_1) - \xi_1) \, ds + (\xi_2 - \xi_1) x + (1-2\theta)(y_2-y_1) \\ = &\int_{y_1}^{y_2} \beta_2(x,t) \, dt +  \int_{(0,x) \times (y_1,y_2)} \operatorname{curl }\beta(s,t) \, dsdt.
\end{align}
We estimate using Lemma \ref{lemma: est W} for $i=1,2$, $c_1 \leq \frac1{576}$ and (1)
\begin{align}
\left| \int_0^x (\beta_1(s,y_i) - \xi_i) \, ds \right| &\leq \left(\int_0^1 |\beta_1(s,y_i) - \xi_i|^2 \, ds \right)^{1/2} \\ &\leq 3 E_{\sigma,\theta,\eps}^{(\epsilon)}(\beta; (0,1) \times \{y_i\})^{1/2} \leq 6 E_{\sigma,\theta,\eps}^{(\epsilon)}(\beta)^{1/2} \leq 6 c_1^{1/2} \theta \leq \frac{\theta}4,
\end{align}
which then implies  together with \eqref{eq: est curl average}
\begin{equation} \label{eq: int beta2}
\left| \int_{y_1}^{y_2} \beta_2(x,t) \, dt - (\xi_2 -\xi_1)x - (1-2\theta)(y_2-y_1) \right| \leq \frac{3\theta}4.
\end{equation}
Now, by Lemma \ref{lemma: alternative} one of the following holds: 
\begin{enumerate}
\item[(3)] $\sigma|\partial_2 \beta_2(x,\cdot)|((y_1,y_2)) \geq \frac{\sigma}2$ for $\epsilon=1$ or $\int_{y_1}^{y_2} (1-\beta_2(x,t)^2)^2 + \sigma^2 |\partial_2 \beta_2(x,t)|^2\,dt \geq \frac{\sigma}{16}$ for $\epsilon =2$; 
\item[(4)] there exists $\bar{\xi} \in \{\pm 1\}$ such that $|\beta_2(x,t) - \bar{\xi}| \leq 3 \min_{\xi' = \pm1 } |\beta_2(x,t) - \xi'|$ holds for all $t \in (y_1,y_2)$.
\end{enumerate}
If (3) holds, we find using (2)
\[
\frac1{16} \min\left\{\theta^2, \sigma, \theta \frac{\sigma^3}{\eps^2}  \right\} \leq \frac{\sigma}{16} \leq E_{\sigma,\theta,\eps}^{(\epsilon)}(\beta; \{x\} \times (y_1,y_2)) \leq 4 E_{\sigma,\theta,\eps}^{(\epsilon)}(\beta).
\]
If (4) holds, we have using \eqref{eq: int beta2} and Lemma \ref{lemma: est W}
\begin{align}
&6 E_{\sigma,\theta,\eps}^{(\epsilon)}(\beta)^{1/2} \\
\geq &3 E_{\sigma,\theta,\eps}^{(\epsilon)}(\beta; \{x\} \times (y_1,y_2))^{1/2} \\
\geq &\left(\int_{y_1}^{y_2} |\beta_2(x,t) - \bar{\xi}|^2 \, dt\right)^{1/2} \\
\geq &\left|\int_{y_1}^{y_2} \beta_2(x,t) - \bar{\xi} \, dt \right| \\
\geq &|(1-\bar{\xi} - 2\theta)(y_2-y_1) - (\xi_2-\xi_1)x| - \frac{3\theta}4 \geq \frac{\theta}4,
\end{align}
where the last estimate follows from the fact that $\bar{\xi},\xi_2,\xi_1 \in \{ \pm 1\}$, $y_2-y_1 = \frac12$ and $x \geq \frac34$.
In particular, we derive that
\[
E_{\sigma,\theta,\eps}^{(\epsilon)}(\beta) \geq \frac{\theta^2}{24^2} \geq \frac1{24^2} \min\left\{\theta^2, \sigma, \theta \frac{\sigma^3}{\eps^2}  \right\}.
\]
This finishes the proof.
\end{proof}

\begin{remark}\label{rem: ani}
    Note that the same proof can also be used to show the same statement for the anisotropic energy $E_{\sigma,\theta,\eps}^{(a)}$.
\end{remark}

Next, we prove a lower bound that includes the term $\sigma \left(\frac{|\log \sigma|}{|\log \theta|} + 1\right)$. For technical reasons, this will be done only for small values of $\sigma$ (relative to $\theta$). In the $\operatorname{curl}$-free case, this gives a more streamlined version of the proof of \cite[Lemma 6]{GinZwi:22}.

\begin{proposition}\label{eq: sigma small}
    Let $k_0 = 32$. Then there exists $c_{LB}^{(3)} > 0$ and  $\theta_0 \in (0,1/2]$ such that it holds for all $0<\theta \leq \theta_0$, $k \in \N$, $k \geq k_0$, $\sigma \in (\theta^{k+1}, \theta^k]$ with $\sqrt{2} \pi \sigma \geq  \eps >0$ and $\epsilon \in \{\pm 1\}$ that
    \[
    \inf_{\mathcal{A}_{\sigma,\theta,\eps}^{(\epsilon)}} E_{\sigma,\theta,\eps}^{(\epsilon)} \geq c_{LB}^{(3)} \min\left\{ \theta^2, k \sigma, \theta \frac{\sigma^3}{\eps^2}  \right\}.
    \]
\end{proposition}

\begin{proof}
\textit{Step 1: Preparation and selection of vertical slices.}  
Let $c_1 := \min\{\frac{c_v}{128}, \frac{1}{4 \cdot 24^2} \}$ and $c_2 = 8$, where $c_v$ is the universal constant from Lemma \ref{lem: lb vortex general}. Additionally, fix $0 < \theta_0 \leq \frac1{16}$ such that it holds for all $k\geq 32$ and $\theta \leq \theta_0$ that $k \theta^{k/4} \leq \theta^2$.
Let $\beta \in \mathcal{A}_{\sigma,\theta,\eps}^{(\epsilon)}$. We may assume that $E_{\sigma,\theta,\eps}^{(\epsilon)}(\beta) \leq c_1 \min\left\{ \theta^2, k \sigma, \theta \frac{\sigma^3}{\eps^2} \right\}$.
Next, find for $i=1,\dots,k$ points $x_i \in (\frac{3\theta^i}4, \theta^i)$ such that
\begin{equation}\label{eq: slice selection}
E_{\sigma,\theta,\eps}^{(\epsilon)}(\beta; \{x_i\} \times (0,1)) \leq 4\theta^{-i} E_{\sigma,\theta,\eps}^{(\epsilon)}\left(\beta; (\frac{3\theta^i}4, \theta^i) \times (0,1) \right).
\end{equation}
We claim the following. \\

\underline{\textit{Claim: }} Let $K = \left\lfloor \frac{k}8 \right\rfloor$. For $i = 1, \dots, K-1$ it holds one of the following for a universal constant $c_3>0$ (not depending on $\beta$, $\theta$, $\sigma$, $\eps$)
\begin{enumerate}
    \item $E_{\sigma,\theta,\eps}^{(\epsilon)}(\beta) \geq c_3 \theta\frac{\sigma^3}{\eps^2}$;
    \item $E_{\sigma,\theta,\eps}^{(\epsilon)}(\beta;\{x_i\} \times (0,1)) \geq c_3 \theta^{-i} \sigma$;
    \item $E_{\sigma,\theta,\eps}^{(\epsilon)}(\beta;(x_{i+1},x_i) \times (0,1)) \geq c_3 \sigma$.
\end{enumerate}

\textit{Step 2: Conclusion using the claim.}
We first show how to conclude the proof using the claim above. Assume that the claim is true. If (1) is true for some $i=1,\dots, K$ then the assertion follows immediately. Hence, we will assume now that for each $i=1,\dots,K-1$ assertion (2) or (3)  from the claim is true. Then we estimate
\begin{align*}
    2 E_{\sigma, \eps}(\beta) &\geq  \sum_{i=1}^{K-1} \left( E_{\sigma,\theta,\eps}^{(\epsilon)}(\beta; (\frac{3\theta^i}4, \theta^i) \times (0,1)) + E_{\sigma,\theta,\eps}^{(\epsilon)}(\beta; (x_{i+1}, x_i) \times (0,1)) \right) \\
    &\geq \sum_{i=1}^{K-1} \left(  \frac{\theta^{i}}4 E_{\sigma,\theta,\eps}^{(\epsilon)}(\beta; \{x_i\} \times (0,1)) + E_{\sigma,\theta,\eps}^{(\epsilon)}(\beta; (x_{i+1},x_i) \times (0,1))\right) \\
    &\geq c_3 (K-1) \sigma \geq c_3\frac{K}2 \sigma \geq \frac{c_3}{32}k \sigma.
\end{align*}
This proves the assertion. \\

\textit{Step 3: Proof of the claim.} 
Exactly as in \eqref{eq: est curl average}, we note that using Lemma \ref{lem: lb vortex general} it holds for a universal constant $c_v > 0$
\begin{equation}\label{eq: est curl whole}
    \int_{(0,1)^2} |\operatorname{curl } \beta| \, d\mathcal{L}^2 \leq \frac{c_1}{c_v} \theta.
\end{equation}
From now on, we will fix $i=1, \dots, K$.

\textit{Step 3a: choice of representative intervals.}
Next, we find an interval $I = (y_1,y_2)$ with $|I| = y_2 - y_1 = c_2 \theta^i$ such that the following holds
\begin{enumerate}
    \item  $\int_{(0,1) \times |I|} |\operatorname{curl } \beta| \, d\mathcal{L}^2 \leq 32|I|  \int_{(0,1)^2} |\operatorname{curl } \beta| \, d\mathcal{L}^2$;
    \item $E_{\sigma,\theta,\eps}^{(\epsilon)}(\beta; \{x_i\} \times I) \leq 32 |I| E_{\sigma,\theta,\eps}^{(\epsilon)}(\beta; \{x_i\} \times (0,1))$;
    \item $E_{\sigma,\theta,\eps}^{(\epsilon)}(\beta; (x_{i+1},x_i) \times I) \leq 32 |I| E_{\sigma,\theta,\eps}^{(\epsilon)}(\beta; (x_{i+1},x_i) \times (0,1))$;
    \item $E_{\sigma,\theta,\eps}^{(\epsilon)}(\beta; (0,1) \times \{y_l\}) \leq 32 E_{\sigma,\theta,\eps}^{(\epsilon)}(\beta)$, for $l=1,2$;
    \item $E_{\sigma,\theta,\eps}^{(\epsilon)}(\beta; (x_{i+1},x_i) \times \{y_l\}) \leq 32 E_{\sigma,\theta,\eps}^{(\epsilon)}(\beta; (x_{i+1},x_i) \times (0,1))$, for $l=1,2$;
\end{enumerate}

\textit{Step 3b: Reduction to one well on vertical and horizontal slices.} We define the one-dimensional function $\alpha: (0,2(x_i - x_{i+1}) + |I|) \to \R^2$ as
\[
\alpha(t) = \begin{cases}
    \beta(x_{i+1} + t,y_1) &\text{ if } t \in (0,x_i - x_{i+1}) \\
    \beta(x_i,y_1 + t - x_i + x_{i+1}) &\text{ if } t \in (x_i - x_{i+1} ,x_i - x_{i+1} +|I|) \\
    \beta(x_i - t + (x_i - x_{i+1} +|I|), y_2) &\text{ if } t \in (x_i - x_{i+1} +|I|, 2(x_i - x_{i+1}) +|I|).
\end{cases}
\]
By Remark \ref{rem: alternative}, one of the following is true
\begin{enumerate}
    \item[(a)] $|\alpha'|((0,2(x_i - x_{i+1}) + |I|)) \geq \frac12$ if $\epsilon = 1$ or $\int_{(0,2(x_i - x_{i+1}) + |I|)} W(\alpha) + \sigma^2 |\alpha'|^2 \, d\mathcal{L}^1 \geq \frac{\sigma}{16}$ if $\epsilon = 2$;
    \item[(b)] there exists $\xi \in K$ such that it holds for almost all $t \in (0,2(x_i - x_{i+1}) + |I|)$ that $|\alpha(t) - \xi| \leq 3\sqrt{2 W(\alpha(t))}$.
\end{enumerate}
If (a) is true, then it follows using (2) and (4) that 
\begin{align*}
\frac{\sigma}{16} &\leq E_{\sigma,\theta,\eps}^{(\epsilon)}(\beta; (x_{i+1},x_i) \times \{y_1\}) + E_{\sigma,\theta,\eps}^{(\epsilon)}(\beta; \{x_i\} \times I) + E_{\sigma,\theta,\eps}^{(\epsilon)}(\beta; (x_{i+1},x_i) \times \{y_2\}) \\
&\leq 64 E_{\sigma,\theta,\eps}^{(\epsilon)}(\beta; (x_{i+1},x_i) \times (0,1)) +  32 |I| E_{\sigma,\theta,\eps}^{(\epsilon)}(\beta; \{x_i\} \times (0,1)).
\end{align*}
Since $|I| = c_2 \theta^i$ this implies that (2) or (3)  from the claim are true for $0< c_3 \leq \frac{1}{2048 c_2}$.
Consequently, we will from now on assume that (b) is true, i.e., there exists $\xi \in K$ such that it holds almost everywhere on $(x_{i+1},x_i) \times \{y_1,y_2\} \cup \{x_i\} \times I$ that
\begin{equation}\label{eq: one well}
|\beta - \xi| \leq 3\sqrt{2W(\beta)}.
\end{equation}

\textit{Step 3c: Conclusion if (b) holds.}
By Lemma \ref{lemma: rewrite}, it holds
\begin{align}
   & \int_{0}^{x_i} \beta_1(s,y_2) - \beta_1(s,y_1) \, ds 
   = \int_{0}^{x_i} \int_{y_1}^{y_2} \operatorname{curl }\beta(s,t) \, ds dt + \int_{y_1}^{y_2} \beta_2(x_i,t) - (1-2\theta)\, dt,
\end{align}
which then implies that
\begin{align}\label{eq: diff beta1}
&\int_{x_{i+1}}^{x_i} \beta_1(s,y_2) - \beta_1(s,y_1) \, ds  \\
   = &\int_{0}^{x_i} \int_{y_1}^{y_2} \operatorname{curl }\beta(s,t) \, ds dt + \int_{y_1}^{y_2} \beta_2(x_i,t)- (1-2\theta)\, dt - \int_{0}^{x_{i+1}} \beta_1(s,y_2) - \beta_1(s,y_1) \, ds.
\end{align}
Next, we will estimate the three terms on the right-hand side. First, note that by (1), \eqref{eq: est curl whole} and $c_1 \leq \frac{c_v}{128}$
\[
\left|\int_{0}^{x_i} \int_{y_1}^{y_2} \operatorname{curl }\beta(s,t) \, ds dt \right| \leq 32|I| \int_{(0,1)^2} |\operatorname{curl } \beta| \, d\mathcal{L}^2 \leq 32\frac{c_1}{c_v} |I| \theta \leq \frac{\theta}4 |I|.
\]
Next, we find invoking \eqref{eq: one well}, Lemma \ref{lemma: est W}, (2) and \eqref{eq: slice selection}
\begin{align}
    \left|\int_{y_1}^{y_2} \beta_2(x_i,t) - (1-2\theta)\, dt \right| &\geq \left| \int_{y_1}^{y_2} \xi_2 - (1-2\theta) \, dt \right| - \int_{y_1}^{y_2} |\beta_2(x_i,t) - \xi_2| \, dt \\
    &\geq 2\theta |I| - 3\sqrt{2} |I|^{1/2} \left( \int_{y_1}^{y_2} W(\beta(x_i,t)) \, dt \right)^{1/2} \\
    &\geq 2\theta |I| - 24  |I| \left(E_{\sigma,\theta,\eps}^{(\epsilon)}(\beta; \{x_i\} \times (0,1))\right)^{1/2} \\
    &\geq 2\theta |I| - 24 |I| \left( 4 c_1 k \sigma \right)^{1/2} \geq \theta |I|,
\end{align}
where we used in the last inequality that $k\sigma \leq k \theta^{k-i} = \theta^2 k \theta^{k-i-2} \leq \theta^2 k \theta^{k/4} \leq \theta^2$ and $0< c_1 \leq \frac{1}{4 \cdot 24^2}$.
Eventually, we estimate for $l=1,2$ invoking (4) and Lemma \ref{lemma: est W}
\begin{align}
    \left|\int_0^{x_{i+1}} \beta_1(s,y_l) \, ds\right| &\leq x_{i+1} + \int_0^{x_{i+1}} \operatorname{dist}(\beta(s,y_l),K) \, ds \\
    &\leq x_{i+1} + x_{i+1}^{1/2} \left(E_{\sigma,\theta,\eps}^{(\epsilon)}(\beta; (0,1) \times \{y_l\})\right)^{1/2} \\
    &\leq \theta^{i+1} + \sqrt{32}  \left(\theta^{i+1}E_{\sigma,\theta,\eps}^{(\epsilon)}(\beta)\right)^{1/2} \\
    &\leq \theta^{i+1} + 8 \left(c_1 k \theta^{i+1} \sigma \right)^{1/2} \leq 2 \theta^{i+1} = \frac{\theta}4 |I|, 
\end{align}
where we used that $|I| = c_2 \theta^i$, $c_2 = 8$, $0< c_1 \leq \frac{1}{64}$ and that $k \theta^{i+1} \sigma \leq \theta^{2i + 2} k\theta^{k-i-1} \leq \theta^{2i+2} k\theta^{k/4} \leq \theta^{2i+2}$.
Combining the three estimates, we obtain from \eqref{eq: diff beta1} that
\begin{equation} \label{eq: difference beta1}
\left| \int_{x_{i+1}}^{x_i} \beta_1(s,y_2) - \beta_1(s,y_1) \, ds \right| \geq \frac{\theta}4 |I|.
\end{equation}
We then estimate using \eqref{eq: one well}, Lemma \ref{lemma: est W} and (5)
\begin{align} 
    &\left| \int_{x_{i+1}}^{x_i} \beta_1(s,y_2) - \beta_1(s,y_1) \, ds \right| \\\leq &\left| \int_{x_{i+1}}^{x_i} \beta_1(s,y_2) - \xi_1 \, ds \right| + \left| \int_{x_{i+1}}^{x_i} \beta_1(s,y_1) - \xi_1 \, ds \right| \\
    \leq &3 \sqrt{2} (x_{i} - x_{i+1})^{1/2} \left( \left(E_{\sigma,\theta,\eps}^{(\epsilon)}(\beta; (x_{i+1},x_{i}) \times \{y_1\})\right)^{1/2} + \left(E_{\sigma,\theta,\eps}^{(\epsilon)}(\beta; (x_{i+1},x_{i}) \times \{y_2\})\right)^{1/2} \right) \\
    \leq &3\sqrt{128}(x_{i} - x_{i+1})^{1/2} \left(E_{\sigma,\theta,\eps}^{(\epsilon)}(\beta; (x_{i+1},x_i) \times (0,1))\right)^{1/2},
\end{align}
from which it follows using $x_i - x_{i+1} \leq \theta^i = \frac{1}{c_2} |I|$ and \eqref{eq: difference beta1} that
\[
E_{\sigma,\theta,\eps}^{(\epsilon)}(\beta; (x_{i+1},x_i) \times (0,1)) \geq \frac{c_2}{16 \cdot 9  \cdot 128} \theta^2 |I| = \frac{c_2^2}{16 \cdot 9 \cdot 128} \theta^{i+2} \geq \frac{2c_2^2}{16 \cdot 9 \cdot 128}  \sigma.
\]
\end{proof}

Eventually, we show how to prove the logarithmic lower bound for $E_{\sigma,\theta,\eps}^{(\epsilon)}$ that stems from the existence of vortices using Proposition \ref{prop: ball constr}.

\begin{proposition}\label{prop: log}
There exists $\theta_0 > 0$ and $c_{LB}^{(4)}>0$ such that it holds for all $\eps, \sigma > 0$ with $0<\theta \leq \theta_0$ and $0 < \frac{1}{\sqrt{2}\pi}\eps  \leq \sigma \leq \frac{\theta}{4}$ it holds 
\[
E_{\sigma,\theta,\eps}^{(\epsilon)}(\beta) \geq c_{LB}^{(4)} \min\{ \theta^2,  \sigma \theta |\log \theta|\}.
\]
\end{proposition}
\begin{proof} 
\textit{Step 1: Preparation.}
Let $\theta_0 \in (0,1/2]$ be such that $|\log \theta_0| \geq \max\{\frac{128}{c_{bc}}, 2 \}$, where $c_{bc} > 0$ is the universal constant from Proposition \ref{prop: ball constr}, and such that $\theta_0 |\log \theta_0| \leq \frac{1}{2\sqrt{2}\pi}$.
Let $\beta \in \mathcal{A}_{\sigma,\theta,\eps}^{(\epsilon)}$ and assume that $E_{\sigma,\theta,\eps}^{(\epsilon)}(\beta) \leq c_1 \min\{ \theta^2, \sigma \theta |\log \theta| \}$, where we define the constant 
\[
c_1 := \min \left\{ \frac{c_v}{\sqrt{2} \cdot \pi \cdot 256 \cdot 64 }, \frac{c_v}{4 \cdot 32^2 \cdot \pi^2}, \frac{1}{32\cdot 320}, \frac{c_{bc}}{512 \cdot 32}, \frac{c_{bc}^2}{432^2 \cdot 4 \cdot 72^2 \cdot 32^2}\right\}
\]
and $c_v, c_{bc}>0$ are the universal constants from Lemma \ref{lem: lb vortex general} and Proposition \ref{prop: ball constr}.

By Lemma \ref{lem: lb vortex general} it follows that $c_v \frac{\sigma^4}{\eps^2} \#\text{ vortices in } (0,1)^2 \leq c_1 \sigma \theta |\log \theta|$ and since $\eps \leq \sqrt{2}\pi \sigma$ we obtain
\[
\#\text{ vortices in } (0,1)^2 \leq 2\pi^2 \frac{c_1}{c_v} \frac{\theta}{\sigma}  |\log \theta| \text{ and } \int_{(0,1)^2} |\operatorname{curl } \beta| \, d\mathcal{L}^2 \leq 2\pi^2 \frac{c_1}{c_v} \theta |\log \theta|.
\]
Let us define $\ell = \min\left\{ \frac{\sigma}{\theta} |\log \theta|, \frac14 \right\}$.
Next, find $x \in (\ell / 2, \ell)$ such that it holds
\begin{equation}\label{eq: choice x}
E_{\sigma,\theta,\eps}^{(\epsilon)}(\beta; \{x\} \times (0,1)) \leq \frac{8}{\ell} E_{\sigma,\theta,\eps}^{(\epsilon)}(\beta) \text{ and } \int_{(x,x+ \frac{\ell}{|\log \theta|}) \times (0,1)} |\operatorname{curl } \beta| \, d\mathcal{L}^2 \leq 16 \pi^2 \frac{c_1}{c_v} \theta  \ell.    
\end{equation}
Then find an interval $I = (y_1,y_2) \subseteq (0,1)$ with $|I| = \ell$ such that $\tilde{I} = \left(y_1 - \frac{\ell}{|\log \theta|},y_2 + \frac{\ell}{|\log \theta|}\right) \subseteq (0,1)$, i.e.~$|\tilde{I}| =  \ell \left( 1 + \frac{2}{|\log \theta|}\right) \leq 2\ell$, it holds that
\begin{enumerate}
    \item $E_{\sigma,\theta,\eps}^{(\epsilon)}(\beta; \{x\} \times I) \leq 32 \ell E_{\sigma,\theta,\eps}^{(\epsilon)}(\beta;\{x\} \times (0,1))$,
    \item $E_{\sigma,\theta,\eps}^{(\epsilon)}(\beta; \{x\} \times \tilde{I}) \leq 64 \ell  E_{\sigma,\theta,\eps}^{(\epsilon)}(\beta;\{x\} \times (0,1))$,
    \item $E_{\sigma,\theta,\eps}^{(\epsilon)}(\beta; (0,1) \times I) \leq 32 \ell E_{\sigma,\theta,\eps}^{(\epsilon)}(\beta)$,
    \item $E_{\sigma,\theta,\eps}^{(\epsilon)}(\beta; (0,1) \times \tilde{I}) \leq 64 \ell E_{\sigma,\theta,\eps}^{(\epsilon)}(\beta)$,
    \item $E_{\sigma,\theta,\eps}^{(\epsilon)}(\beta; (0,1) \times \{y_i\}) \leq 32 E_{\sigma,\theta,\eps}^{(\epsilon)}(\beta)$,
    \item $\#\text{ vortices in } (0,1) \times \tilde{I} \leq 32\cdot 32 \pi^2 \ell \frac{c_1}{c_v} \frac{\theta}{\sigma} |\log \theta|$, 
    \item $\int_{(0,1) \times (\tilde{I} \setminus I)} |\operatorname{curl } \beta| \, d\mathcal{L}^2 \leq 32 \cdot 32 \pi^2\cdot \ell \frac{c_1}{c_v} \theta$,
    \item $\int_{(x,x+\frac{\ell}{|\log \theta|}) \times I} |\operatorname{curl } \beta| \, d\mathcal{L}^2 \leq 32 \cdot 16 \pi^2 \ell^2 \frac{c_1}{c_v} \theta \leq 32 \cdot 4\pi^2 \ell \frac{c_1}{c_v} \theta$.
\end{enumerate}
By (1), (5) and \eqref{eq: choice x} we observe for $\epsilon =1 $ that 
\[
\sum_{i=1}^2 \sigma |\partial_1 \beta(\cdot, y_i)|((0,x]) + \sigma |\partial_2 \beta(x,\cdot)|(I) \leq 320 E_{\sigma,\theta,\eps}^{(\epsilon)}(\beta) \leq 320 c_1 \sigma \theta |\log \theta| \leq 320 c_1 \sigma
\]
and for $\epsilon =2$ similarly that
\[
\sum_{i=1}^2 \int_{0}^x W(\beta(s,y_i)) + \sigma^2 |\partial_1 \beta(s,y_i)|^2 \, ds + \int_{I} W(\beta(x,t)) + \sigma^2 |\partial_1 \beta(x,t)|^2 \, dt \leq 320 c_1 \sigma.
\]
Since $c_1 \leq \frac{1}{32\cdot 320}$ it follows from Remark \ref{rem: alternative} that there exists $\xi \in K$ such that it holds for all $s \in (0,x)$ and $t \in I$
\begin{equation}\label{eq: one well 1}
|\beta(s,y_i) - \xi| \leq 3\sqrt{2} \operatorname{dist}(\beta(s,y_i),K) \text{ and } |\beta(x,t) - \xi| \leq 3\sqrt{2} \operatorname{dist}(\beta(x,t),K).
\end{equation}
Using Lemma \ref{lemma: rewrite} we write
\begin{align}
    \int_I \left(\beta_2(x,t) - (1-2\theta)\right) \, dt = \int_{y_1}^{y_2} \operatorname{curl} \beta(s,t) \, ds dt + \int_{0}^x (\beta_1(s,y_2) - \xi_1) - (\beta_1(s,y_1) - \xi_1) \, ds. \label{eq: form in beta2}
\end{align}
Then note that by \eqref{eq: one well 1}, Lemma \ref{lemma: est W} and (5) it holds for $i=1,2$
\begin{align}
\left| \int_{0}^x (\beta_1(s,y_i) - \xi_1) \, ds \right| &\leq 3\sqrt{2} \int_0^x \operatorname{dist}(\beta(s,y_i),K) \, ds \\ &\leq x^{1/2} E_{\sigma,\theta,\eps}^{(\epsilon)}(\beta; (0,1) \times \{y_i\})^{1/2} \\
&\leq x^{1/2} \left( 32 c_1 \min\{ \theta^2, \sigma \theta |\log \theta|\}\right)^{1/2} \leq \frac14 \theta \ell = \frac{\theta}4 (y_2-y_1), \label{eq: est line x}
\end{align}
where we used that $x  \leq \ell = (y_2-y_1)$, $\min\{ \theta^2, \sigma \theta |\log \theta|\} \leq 4 \theta^2 \ell$ and $c_1 \leq \frac1{4\cdot 32 \cdot 16}$.

Next, we distinguish two cases depending on the size of $ \int_{y_1}^{y_2} \operatorname{curl} \beta(s,t) \, ds dt$.  \\
\textit{Step 2: The case of a small $\operatorname{curl}$.}
Assume that 
\[ 
\left| \int_{y_1}^{y_2} \operatorname{curl} \beta(s,t) \, ds dt\right| \leq \theta (y_2 - y_1).
\]
Then it follows from (1) , \eqref{eq: one well 1}, \eqref{eq: form in beta2}, \eqref{eq: est line x} and Lemma \ref{lemma: est W} that
\begin{align}
2\theta (y_2-y_1) &\leq \left| (1-\xi_2-2\theta)(y_2-y_1) \right| \\
&\leq \left| \int_I \beta_2(x,t) - \xi_2 \, dt \right|  + \left| \int_{I} \beta_2(x,t) - (1-2\theta) \, dt \right| \\
&\leq 3\sqrt{2} (y_2 - y_1)^{1/2} \left( \int_I \operatorname{dist}(\beta(x,t),K)^2 \, dt \right)^{1/2}+ \frac{3}2 \theta (y_2-y_1) \\
&\leq 3\sqrt{2} (y_2-y_1)^{1/2} E_{\sigma,\theta,\eps}^{(\epsilon)}(\beta; \{x\} \times I)^{1/2} + \frac32 \theta (y_2-y_1) \\
&\leq 24 (y_2-y_1) E_{\sigma,\theta,\eps}^{(\epsilon)}(\beta; \{x\} \times (0,1))^{1/2} + \frac32 \theta (y_2-y_1),
\end{align}
which in turn implies by \eqref{eq: choice x}
\[
E_{\sigma,\theta,\eps}^{(\epsilon)}(\beta) \geq \frac{\ell}{8}  E_{\sigma,\theta,\eps}^{(\epsilon)}(\beta;\{x\} \times (0,1)) \geq \frac{1}{8 \cdot 48^2} \ell \theta^2 =  \frac{1}{4 \cdot 8 \cdot 48^2} \min\{ \theta^2,  \sigma \theta |\log \theta|\} .
\]
\textit{Step 2: The case of a large $\operatorname{curl}$.} Now, we assume that 
\begin{equation} \label{eq: curl large}
\left| \int_{(0,x) \times I} \operatorname{curl }\beta \, d\mathcal{L}^2 \right| \geq \theta (y_2-y_1).
\end{equation}
For technical reasons we will provide different proofs for $\epsilon=1$ and $\epsilon =2$.\\

\textit{Step 2a: The case $\epsilon = 1$.}
Define $A' := (0,x) \times I \subseteq A := (-\frac{\sigma}{\theta }, x + \frac{\sigma}{\theta }) \times (y_1 - \frac{\ell}{|\log \theta|} + \eps, y_2 + \frac{\ell}{|\log \theta|} - \eps)$ and $d:= \operatorname{dist}(A',\partial A) = \frac{\ell}{|\log \theta|} - \eps$. 
Note that it holds $\frac{\ell}{|\log \theta|} - \eps \geq \frac{\sigma}{\theta} - \eps \geq \frac{\sigma}{2\theta} + \sqrt{2}\pi \sigma- \eps \geq \frac{\sigma}{2\theta}$ since we assume $\theta \leq \theta_0 \leq \frac{1}{2\sqrt{2}\pi}$ and $\eps \leq \sqrt{2}\pi \sigma$.

Next, extend $\beta$ by $(\xi_1,1-2\theta)$ to $(-\frac{\sigma}{\theta},0) \times \tilde{I}$ and let $x_1, \dots, x_n \in (0,1) \times \tilde{I}$ such that it holds $\operatorname{curl } \beta = \sum_{k=1}^n \gamma_k \sigma \delta_{x_k} * \rho_{\eps}$ with $\gamma_k \in \{\pm 1\}$ in $A$. 
Using (6) it holds that $n \leq 32^2 \pi^2 \ell \frac{c_1}{c_v} \frac{\theta}{\sigma} |\log \theta| \leq 32^2 \pi^2 \frac{c_1}{c_v} |\log \theta|^2$.
Eventually, set $T := \log\left( \frac{d}{\eps n} \right)$ and note that it holds since $\eps \leq \sqrt{2} \pi \sigma$ and $c_1 \leq \frac{c_v}{\sqrt{2}\cdot \pi \cdot 256 \cdot 64} $ that
\[
\log\left( \frac{d}{\eps n} \right) \geq \log\left( \frac{\sigma}{\theta} \frac{c_v }{\eps 32^2 \cdot c_1   |\log \theta|^2} \right) \geq |\log \theta| - \log\left(  \frac{|\log \theta|^2}{16} \right)  \geq \frac12 |\log \theta|,
\]
where we used that $\log\left(  \frac{|\log \theta|^2}{16} \right) =2 \log \left(  \frac{|\log \theta|}{4} \right) \leq \frac12 |\log \theta|$.
Then it follows in the notation from Proposition \ref{prop: ball constr} that 
\begin{align}\label{eq: bc est lower}
    \int_A W(\beta) \, d\mathcal{L}^2 + \sigma |D\beta|(A) \geq \frac{c_{bc}}2 \sigma |\log \theta| \left|\operatorname{curl } \beta(\bigcup_{j \in J(T)} B_{r_j(T)}(p_j(T)))\right|.
\end{align}
On the other hand, we estimate using (4) and $|\log \theta| \geq |\log \theta_0| \geq \frac{128}{c_{bc}}$
\begin{align}
\int_A W(\beta) \, d\mathcal{L}^2 \leq \int_{(0,1) \times \tilde{I}} W(\beta) \, d\mathcal{L}^2 + 4 \theta^2 \frac{\sigma}{\theta} |\tilde{I}| &\leq 64 \ell E_{\sigma,\theta,\eps}^{(1)}(\beta) + 8 \sigma \theta \ell \\
&\leq 64 \ell E_{\sigma,\theta,\eps}^{(1)}(\beta) + \frac{c_{bc}}{16} \sigma \theta |\log \theta| \ell \label{eq: est bulk eps1}
\end{align}
and using (2), (4), \eqref{eq: choice x}, \eqref{eq: one well 1}, Lemma \ref{lemma: est W}, $c_1 \leq \min\{ \frac{c_{bc}}{512 \cdot 32}, \frac{c_{bc}^2}{512 \cdot 6^2 \cdot 32^2}\}$
\begin{align}
\sigma |D\beta|(A) &\leq \sigma |D\beta|((0,1) \times \tilde{I}) + \sigma \int_{\tilde{I}} |\beta_1(0,t) - \xi_1| \, dt \\
&\leq 64 \ell  E_{\sigma,\theta,\eps}^{(1)}(\beta) + \sigma \int_{\tilde{I}} |\beta_1(0,t) - \beta_1(x,t)| \, dt + \sigma \int_{\tilde{I}} |\beta_1(x,t) -\xi_1| \, dt \\
&\leq 64 \ell  E_{\sigma,\theta,\eps}^{(1)}(\beta) + \sigma |D\beta|((0,1) \times \tilde{I} ) + 3\sqrt{2} \sigma |\tilde{I}|^{1/2} \left( \int_{\tilde{I}} \operatorname{dist}(\beta(x,t),K)^2 \, dt \right)^{1/2} \\
&\leq 128 \ell E_{\sigma,\theta,\eps}^{(1)}(\beta) + 3\sqrt{2} \sigma |\tilde{I}|^{1/2} \left(E_{\sigma,\theta,\eps}^{(1)}(\beta; \{x\} \times \tilde{I}) \right)^{1/2} \\
&\leq 128 \ell c_1 \min\{ \theta^2,\sigma \theta |\log \theta|\} + 6\sqrt{512} \sigma  \ell^{1/2} \left(E_{\sigma,\theta,\eps}^{(1)}(\beta)\right)^{1/2} \\
&\leq 512 c_1 \theta^2 \ell^2   + 6 \sqrt{512 c_1} \sigma \theta \ell \\
&\leq \frac{c_{bc}}{32} \theta^2 \ell^2 + \frac{c_{bc}}{32} \sigma \theta \ell \leq \frac{c_{bc}}{16} \theta^2 \ell^2. \label{eq: est Dbeta}
\end{align}
In the last inequality, we used that $\theta \ell = \min\{ \theta/4, \sigma |\log \theta|\} \geq \sigma$. 
Moreover, by (7) and (8) we obtain since $c_1 \leq \frac{c_v}{4 \cdot 32^2 \cdot \pi^2}$
\begin{align}
    \left|\operatorname{curl } \beta(\bigcup_{B_{r_i(T)}(p_i(T)) \cap A' \neq \emptyset} B_{r_i(T)}(p_i(T)))\right| &\geq \left| \int_{A'} \operatorname{curl} \beta \, d\mathcal{L}^2 \right| - \int_{A \setminus A'} |\operatorname{curl} \beta| \, d\mathcal{L}^2 \\
    &\geq  \theta (y_2-y_1) - 2 \cdot 32^2 \pi^2\cdot \ell \frac{c_1}{c_v} \theta \geq \frac{\theta}2 \ell. \label{eq: curl difference}
\end{align}
Combining \eqref{eq: bc est lower}, \eqref{eq: est bulk eps1}, \eqref{eq: est Dbeta}, and \eqref{eq: curl difference} yields
\begin{align}
    \frac{c_{bc}}4 \sigma \theta |\log \theta| \ell &\leq \frac{c_{bc}}2 \sigma |\log \theta|  \left|\operatorname{curl } \beta(\bigcup_{B_{r_i(T)}(p_i(T)) \cap A' \neq \emptyset} B_{r_i(T)}(p_i(T)))\right| \\
    &\leq \int_A W(\beta) \, d\mathcal{L}^2 + \sigma |D\beta|(A) \\
    &\leq   64 \ell E_{\sigma,\theta,\eps}^{(1)}(\beta) + \frac{c_{bc}}{16} \sigma \theta  |\log \theta| \ell + \frac{c_{bc}}{16} \theta^2 \ell^2. \label{eq: put together}
\end{align}
If $\ell = \frac{\sigma}{\theta} |\log \theta|$, we find $\sigma \theta |\log \theta| \ell = \theta^2 \ell^2 = \sigma^2 |\log \theta|^2$ and therefore
\[
\frac{c_{bc}}{512} \sigma \theta |\log \theta| \leq E_{\sigma,\theta,\eps}^{(1)}(\beta).
\]
On the other hand, if $\ell = \frac14$ then $\frac{\sigma}{\theta} |\log \theta| \geq \frac14 = \ell$ which implies that $\sigma \theta |\log \theta| \ell \geq \theta^2 \ell^2$ and consequently we obtain again
\[
\frac{c_{bc}}{512} \sigma \theta |\log \theta| \leq E_{\sigma,\theta,\eps}^{(1)}(\beta).
\]
This finishes the proof in the case $\epsilon = 1$.\\
\textit{Step 3b: The case $\epsilon = 2$.}
We define $A$ and $A'$ as in the case $\epsilon = 1$.
Again, we extend $\beta$ as $(\xi_1,1-2\theta)$ to $(-\infty,0) \times \tilde{I}$.
Next, we estimate 
\[
2 \int_{(0,1) \times \tilde{I}} W(\beta) + \sigma^2 |D\beta|^2 \, d\mathcal{L}^2 \geq \int_{(0,1) \times \tilde{I}} W(\beta) + 2 \sigma |D (\Phi \circ \beta_1)| + 2 \sigma |D (\Phi \circ \beta_2)| \, d\mathcal{L}^2.
\]
Then we can estimate similarly to \eqref{eq: est Dbeta} using (2), (4), Lemma \ref{lem: properties phi}, \eqref{eq: choice x}, \eqref{eq: one well 1}, Lemma \ref{lemma: est W} and $c_1 \leq \min\{ \frac{c_{bc}}{512 \cdot 32}, \frac{c_{bc}^2}{432^2 \cdot 4 \cdot 72^2 \cdot 32^2}\}$ (we assume wlog that $0 < c_{bc} < 1$) 
\begin{align}
&2\sigma |D(\Phi \circ \beta_1)|(A) \\ &\leq 2\sigma \int_{(0,1) \times \tilde{I}} |D(\Phi \circ \beta_1)| \, d \mathcal{L}^2 + 2\sigma \int_{\tilde{I}} |\Phi(\beta_1(0,t)) - \Phi(\xi_1)| \, dt \\
&\leq   E_{\sigma,\theta,\eps}^{(2)}(\beta; (0,1) \times \tilde{I}) + 2 \sigma \int_{\tilde{I}} |\Phi(\beta_1(0,t)) - \Phi(\beta_1(x,t))| + |\Phi(\beta_1(x,t)) - \Phi(\xi_1)| \,dt \\
&\leq   E_{\sigma,\theta,\eps}^{(2)}(\beta; (0,1) \times \tilde{I}) + 2 \sigma \int_{(0,1) \times \tilde{I}} |D(\Phi \circ \beta_1)| \, d\mathcal{L}^2  + 2 \sigma\int_{\tilde{I}} |\Phi(\beta_1(x,t)) - \Phi(\xi_1)| \,dt \\
&\leq 2  E_{\sigma,\theta,\eps}^{(2)}(\beta; (0,1) \times \tilde{I}) + 8 \sigma \int_{\tilde{I}} |\beta_1(x,t) - \xi_1| + |\beta_1(x,t) - \xi_1|^3 \,dt \\
&\leq 128 \ell c_1 \min\{ \theta^2 , \sigma \theta |\log \theta| \} + 8 \sigma |\tilde{I}|^{1/2}  \left( \int_{\tilde{I}} |\beta_1(x,t) - \xi_1|^2 \, dt \right)^{1/2} + 8 \sigma |\tilde{I}|^{1/4} \left( \int_{\tilde{I}} |\beta_1(x,t) - \xi_1|^4 \, dt \right)^{3/4} \\
& \leq 512 c_1 \ell^2 \theta^2  + 48\sigma  \ell^{1/2} \left(E_{\sigma,\theta,\eps}^{(2)}(\beta; \{x\} \times \tilde{I})\right)^{1/2} + 432 \sigma \ell^{1/4} \left(E_{\sigma,\theta,\eps}^{(2)}(\beta; \{x\} \times \tilde{I})\right)^{3/4}  \\
&\leq \frac{c_{bc}}{32} \theta^2 \ell^2  + 48  \sigma \ell^{1/2} \left( 512 c_1 \min\{\theta^2,\sigma \theta |\log \theta|\} \right)^{1/2} + 432  \sigma \ell^{1/4} \left( 512 c_1 \min\{\theta^2,\sigma \theta |\log \theta|\} \right)^{3/4} \\
&\leq \frac{c_{bc}}{32} \theta^2 \ell^2 + 72\sqrt{2048 c_1} \sigma \theta \ell + 432  (c_1 2048)^{3/4} \sigma \theta^{3/2} \ell \\
&\leq \frac{c_{bc}}{32} \theta^2 \ell^2 + \frac{c_{bc}}{32} \sigma \theta \ell \leq \frac{c_{bc}}{16} \theta^2 \ell^2.
\end{align}

Then we can estimate similarly to the case $\epsilon = 1$
\[
\int_A W(\beta) + 2 \sigma |D(\Phi \circ \beta_1)| + 2 \sigma |D(\Phi \circ \beta_2)| \, d\mathcal{L}^2 \leq 64 \ell E_{\sigma,\theta,\eps}^{(1)}(\beta) + \frac{c_{bc}}{16} \sigma \theta |\log \theta| \ell + \frac{c_{bc}}{16} \theta^2 \ell^2.
\]
Using Proposition \ref{prop: ball constr} it follows that
\[
\int_{A} W(\beta) + 2 \sigma |D(\Phi \circ \beta_1)| + 2 \sigma |D(\Phi \circ \beta_2)| \,  d\mathcal{L}^2 \geq \frac{c_{bc}}2 \sigma |\log \theta|  \left|\operatorname{curl } \beta(\bigcup_{B_{r_i(T)}(p_i(T)) \cap A' \neq \emptyset} B_{r_i(T)}(p_i(T)))\right|.
\]
Then we argue exactly as in the case $\epsilon = 1$ in \eqref{eq: curl difference} and \eqref{eq: put together} to conclude.

\end{proof}

Eventually, we are able to provide the lower bound from Theorem \ref{thm: main} in the case $\epsilon = 1,2$.

\begin{proposition}\label{prop: main lower bound}
    There exists $c>0$ such that it holds for all $\theta \in (0,1/2]$, $\sigma \geq \sqrt{2} \pi \eps >0$ and $\epsilon \in \{1,2\}$ that
    \[
    \inf_{\mathcal{A}_{\theta,\sigma,\eps}^{(\epsilon)} } E_{\sigma,\theta,\eps}^{(\epsilon)} \geq c \min\left\{ \theta^2, \sigma \left( \frac{|\log \sigma|}{|\log \theta|} + 1 \right), \theta \frac{\sigma^3}{\eps^2} + \theta \sigma |\log \theta| \right\}.
    \]
\end{proposition}

\begin{proof}
    Let $\theta_0 \in (0,1/2]$ be such that Proposition \ref{eq: sigma small} and Proposition \ref{prop: log} hold for all $0< \theta \leq \theta_0$.
    Now we distinguish the following cases:
    First assume that $\theta_0 \leq \theta$. 
    Note that for $\frac12 \geq \theta \geq \theta_0$ it holds 
        \[
        \theta \sigma |\log \theta| \leq \theta \sigma |\log \theta_0| \leq  |\log \theta_0| \theta\frac{\sigma^3}{\eps^2}.
        \]
        Hence, by Proposition \ref{prop: theta large} we estimate for all $\sigma > \sqrt{2}\pi \eps$
        \begin{align}
        \inf_{\mathcal{A}_{\sigma,\theta,\eps}^{(\epsilon)}}  E_{\sigma,\theta,\eps}^{(\epsilon)}  \geq &c_{LB}^{(1)} \min\left\{ \theta^2, \sigma (|\log \sigma|+1), \frac{\sigma^3}{\eps^2}  \right\} \\ \geq &c_{LB}^{(1)} \log(2) \min\left\{ \theta^2, \sigma \left(\frac{|\log \sigma|}{|\log \theta|}+1 \right), \theta \frac{\sigma^3}{\eps^2} \right\} \\ \geq &\frac{c_{LB}^{(1)} \log(2)}{2 |\log (\theta_0)|} \min\left\{ \theta^2, \sigma \left(\frac{|\log \sigma|}{|\log \theta|}+1 \right), \theta \frac{\sigma^3}{\eps^2} + \theta \sigma |\log \theta| \right\}.
        \end{align}
        Next, we assume that $0 < \theta \leq \theta_0$. We distinguish the following cases:
        \begin{itemize}
        \item If $\sigma \geq 1$ then it follows from Proposition \ref{prop: sigma simple} that
        \[
        \inf_{\mathcal{A}_{\sigma,\theta,\eps}^{(\epsilon)}}E_{\sigma,\theta,\eps}^{(\epsilon)} \geq c_{LB}^{(2)} \min\{ \theta^2, \sigma, \theta \frac{\sigma^3}{\eps^2} \} = c_{LB}^{(2)} \theta^2 \geq c_{LB}^{(2)} \min\left\{ \theta^2,\sigma \left( \frac{|\log \sigma|}{|\log \theta|} + 1 \right), \theta \frac{\sigma^3}{\eps^2}  \right\}.
        \]
            \item If $1 > \sigma \geq \theta^{32}$ then it holds that $\frac{|\log \sigma|}{|\log \theta|} \leq 32$. Consequently, using Proposition \ref{prop: sigma simple} we estimate
            \[
            \inf_{\mathcal{A}_{\sigma,\theta,\eps}^{(\epsilon)}}E_{\sigma,\theta,\eps}^{(\epsilon)} \geq c_{LB}^{(2)} \min\left\{ \theta^2, \sigma, \theta \frac{\sigma^3}{\eps^2}  \right\} \geq \frac{c_{LB}^{(2)}}{64} \min\left\{ \theta^2, \sigma \left( \frac{|\log \sigma|}{|\log \theta|} + 1 \right), \theta \frac{\sigma^3}{\eps^2}\right\}.
            \]
            \item If $\sigma \in (\theta^{k+1},\theta^k]$ for some $k \geq 32$ then we have that $\frac{|\log \sigma|}{|\log \theta|} \leq k+1 \leq 2k$ which implies using Proposition \ref{prop: log}
            \[
            \inf_{\mathcal{A}_{\sigma,\theta,\eps}^{(\epsilon)}}E_{\sigma,\theta,\eps}^{(\epsilon)}  \geq c_{LB}^{(3)} \min\left\{ \theta^2, k \sigma, \theta \frac{\sigma^3}{\eps^2}  \right\} \geq \frac{c_{LB}^{(3)}}4 \min\left\{ \theta^2,  \sigma \left( \frac{|\log \sigma|}{|\log \theta|} + 1\right), \theta \frac{\sigma^3}{\eps^2}  \right\}.
            \]
        \end{itemize}
        Summarizing the above cases, we find for $c := \min\left\{ \frac{c_{LB}^{(2)}}{64}, \frac{c_{LB}^{(3)}}{4}  \right\}$ that it holds for $0<\theta \leq \theta_0$ and $ \sigma > \sqrt{2}\pi \eps > 0$ that
        \begin{equation} \label{eq: lower theta small total}
        \inf_{\mathcal{A}_{\sigma,\theta,\eps}^{(\epsilon)}}E_{\sigma,\theta,\eps}^{(\epsilon)}  \geq c \min\left\{ \theta^2,  \sigma \left( \frac{|\log \sigma|}{|\log \theta|} + 1\right), \theta \frac{\sigma^3}{\eps^2}  \right\}.
        \end{equation}
        Now, note that if $\sigma \geq \frac{\theta}4$ it holds that $\theta \frac{\sigma^3}{\eps^2} \geq \frac{\theta^2}4$ and therefore
        \[
        \inf_{\mathcal{A}_{\sigma,\theta,\eps}^{(\epsilon)}}E_{\sigma,\theta,\eps}^{(\epsilon)}  \geq c \min\left\{ \theta^2,  \sigma \left( \frac{|\log \sigma|}{|\log \theta|} + 1\right), \theta \frac{\sigma^3}{\eps^2}  \right\} \geq \frac{c}4 \min\left\{ \theta^2,  \sigma \left( \frac{|\log \sigma|}{|\log \theta|} + 1\right), \theta \frac{\sigma^3}{\eps^2 } +\theta \sigma |\log \theta| \right\}.
        \]
        On the other hand, if $\sigma \leq \frac{\theta}4$ it follows from \eqref{eq: lower theta small total} and Proposition \ref{prop: log}
        \begin{align}
            \inf_{\mathcal{A}_{\sigma,\theta,\eps}^{(\epsilon)}}E_{\sigma,\theta,\eps}^{(\epsilon)}  \geq& \frac{c}2 \min\left\{ \theta^2,  \sigma \left( \frac{|\log \sigma|}{|\log \theta|} + 1\right), \theta \frac{\sigma^3}{\eps^2}  \right\} + \frac{c_{LB}^{(4)}}2 \min\{\theta^2, \sigma \theta |\log \theta|\} \\ \geq &\min\{ c/2, c_{LB}^{(4)}/2 \} \, \min\left\{ \theta^2, \sigma \left( \frac{|\log \sigma|}{|\log \theta|} + 1 \right), \theta \frac{\sigma^3}{\eps^2} + \theta \sigma |\log \theta| \right\}.
        \end{align}
\end{proof}

\subsection{The lower bound for the anisotropic quadratic perturbation}

We start with an estimate that controls the full gradient by the derivatives $\partial_1 \beta_1$ and $\partial_2 \beta_2$. In the case of $\operatorname{curl}$-free fields, this essentially resembles a classical estimate for the Laplacian.

\begin{lemma} \label{lemma: est elliptic}
There exists $C>0$ such that it holds for all $0<A<L,H$ and $\beta \in L^2((-A,L+A)\times (-A,H+A);\R^2)$ with $\partial_{1} \beta_1, \partial_2 \beta_2, \operatorname{curl} \beta \in L^2((-A,L+A)\times (-A,H+A)$ that
\begin{align}
&\int_{(0,L) \times (0,H)} |D \beta|^2 \, d\mathcal{L}^2 \leq  &C \left(\int_{(-A,L+A)\times (-A,H+A)} (\partial_1 \beta_1)^2 + (\partial_2 \beta_2)^2 + A^{-2} |\beta|^2 + |\operatorname{curl }\beta|^2  \, dx \right)
\end{align}
\end{lemma}

\begin{proof}
By a density argument, we may assume that $\beta$ is smooth. Let $\eta \in C^{\infty}_c((-A,L+A) \times (-A,H+A); [0,1])$ be such that $\eta= 1$ on $(0,L) \times (0,H)$ and $|\nabla \eta| \leq CA^{-1}$. Then we find using integration by parts
\begin{align}
&\int_{(-A,L+A)\times (-A,H+A)} \left( \partial_2 \beta_2 \cdot \eta^2 \partial_1 \beta_1 + \partial_2 \beta_2 \cdot \beta_1 2 \eta \partial_1 \eta \right) \, d\mathcal{L}^2 \\
=&\int_{(-A,L+A)\times (-A,H+A)} \partial_2 \beta_2 \cdot \partial_1(\eta^2 \beta_1) \, d\mathcal{L}^2 \\
 = &\int_{(-A,L+A)\times (-A,H+A)} \partial_1 \beta_2 \cdot \partial_2 (\eta^2 \beta_1) \, d\mathcal{L}^2 \\
= &\int_{(-A,L+A)\times (-A,H+A)} \left(\eta^2 |\partial_2 \beta_1|^2  + (\operatorname{curl }\beta) \cdot \eta^2 (\partial_2 \beta_1) + \partial_2 \beta_1 (2\eta \partial_2 \eta) \beta_1 + (\operatorname{curl }\beta) \cdot \beta_1 2 \eta \partial_2 \eta \right) \, d\mathcal{L}^2 
\end{align}
Using Young's inequality, we obtain
\begin{align}
        &\int_{(-A,L+A)\times (-A,H+A)} \eta^2 |\partial_2 \beta_1|^2 \, d\mathcal{L}^2 \\
    \leq &  \int_{(-A,L+A)\times (-A,H+A)}  2 |\operatorname{curl } \beta|^2 + 7 |\beta|^2 |\nabla \eta|^2 + |\partial_2 \beta_2|^2 + |\partial_1 \beta_1|^2 + \frac12 \eta^2 |\partial_2 \beta_1|^2 \, d\mathcal{L}^2 \\
\end{align}
and therefore by absorbing the last term on the left-hand side
\begin{align}
&\int_{(0,L)\times (0,H)}  |\partial_2 \beta_1|^2 \, d\mathcal{L}^2 \\
    \leq &\int_{(-A,L+A)\times (-A,H+A)} \eta^2 |\partial_2 \beta_1|^2 \, d\mathcal{L}^2 \\
    \leq &C \int_{(-A,L+A)\times (-A,H+A)}  \bigg( |\operatorname{curl } \beta|^2 + A^{-2} |\beta|^2  + |\partial_2 \beta_2|^2 + |\partial_1 \beta_1|^2 \bigg)\, d\mathcal{L}^2.
\end{align}
The estimate for $\int_{(0,L)\times (0,H)}  |\partial_1 \beta_2|^2 \, d\mathcal{L}^2 $ is shown analogously.
\end{proof}

The above estimate will allow us to control $E_{\sigma,\theta,\eps}^{(a)}$ on a small domain from below by $E_{\sigma,\theta,\eps}^{(2)}$ on a larger domain. 

\begin{lemma}\label{lemma: est anisotropic}
    There exists $c>0$ such that for all $0 < \eps < \sqrt{2} \pi \sigma \leq \sqrt{2} \pi$ it holds for all $\beta \in \mathcal{A}_{\sigma,\theta,\eps}^{(a)}$
    \[
    c \int_{(0,1/2) \times (1/4,3/4)} W(\beta) + \sigma^2 |D\beta|^2 \, d\mathcal{L}^2 \leq \int_{(0,1)^2} W(\beta) + \sigma^2 |\partial_1 \beta_1|^2 + \sigma^2|\partial_2 \beta_2|^2 \, d\mathcal{L}^2 + \sigma^2.
    \]
\end{lemma}
\begin{proof}
    Let $\beta \in \mathcal{A}^{(a)}_{\sigma,\theta,\eps}$ with $\operatorname{curl} \beta = \sigma \sum_{k=1}^n \gamma_k \delta_{x_k} * \rho_{\eps}$.
    We define the sets $S_k = \{ x_k + t e_1: t > 0\}$ and set
    \[
    \mu = \sigma \sum_{k=1}^n \gamma_k e_2 \mathcal{H}^1_{| S_k } * \rho_{\eps}. 
    \]
    Then it holds that $\operatorname{curl } \mu = \operatorname{curl } \beta$.
    In particular, there exists $u \in W^{1,2}((0,1)^2)$ such that $\nabla u = \beta - \mu$. 
    As $\mu$ is smooth, $\partial_1\partial_1 u, \partial_2\partial_2 u \in L^2((0,1)^2)$ exist.
    Next, we will extend $u$ to a function $\tilde{u} \in W^{1,2}((-1,1) \times (0,1))$ such that $\partial_1\partial_1 \tilde{u}, \partial_2 \partial_2 \tilde{u} \in L^2((-1,1) \times (0,1))$ exist, cf.~\cite[Lemma 3.2]{GiKoZw}.
    For this let $c_1,c_2,c_3 \in \R$ be the unique solutions to the linear system
    \begin{equation}\label{eq: linear system}
        \sum_{i=1}^3 c_i = 1 , \qquad -\sum_{i=1}^3 c_i \frac{i}6 = 1 \qquad \text{ and } \qquad \sum_{i=1}^3 c_i \left(\frac{i}6\right)^2 = 1.
    \end{equation}
    Then we define 
    \[
    \tilde{u}(x,y) = \begin{cases}
        u(x,y) & \text{ if } x \geq 0, \\
        \sum_{i=1}^3 c_i u(-\frac{i}6 x, y) &\text{ if } x < 0.
    \end{cases}
    \]
    It follows from the equations \eqref{eq: linear system} that $\tilde{u} \in W^{1,2}((-1,1)\times (0,1))$ with $\partial_1\partial_1 \tilde{u}, \partial_2 \partial_2 \tilde{u} \in L^2((-1,1) \times (0,1))$.
    Next, we define 
    \[
    \tilde{\beta}(x,y) = \begin{cases}
        \beta(x,y) &\text{ if } x \geq 0, \\
        -\sum_{i=1}^3 c_i \frac{i}6 \beta(-\frac{i}6 x,y) &\text{ if } x < 0.
    \end{cases} \quad \text{ and } \quad  \tilde{\mu}(x,y) = \begin{cases}
        \mu(x,y) &\text{ if } x \geq 0, \\
        -\sum_{i=1}^3 c_i \frac{i}6 \mu(-\frac{i}6 x,y) &\text{ if } x < 0.
    \end{cases}
    \]
    Then it follows $\nabla \tilde{u} = \tilde{\beta} - \tilde{\mu}$. 
    Since $\partial_1\partial_1 \tilde{u}, \partial_2\partial_2 \tilde{u}$ and $\tilde{\mu}$ is smooth, it follows that $\partial_1 \tilde{\beta}_1, \partial_2 \tilde{\beta}_2 \in L^2((-1,1)\times (0,1))$ exist and for $x<0$ has the form $\partial_1 \tilde{\beta}_1(x,y) = \sum_{i=1}^3 c_i \left(\frac{i}6\right)^2 \partial_1 \beta_1(-\frac{i}6x,y)$ and $\partial_2 \tilde{\beta}_2(x,y) = -\sum_{i=1}^3 c_i \frac{i}6 \partial_2 \beta_2(-\frac{i}6x,y)$.
    Next, we compute for $x < 0$
    \begin{align}
        \left( \mathcal{H}^1_{|S_k} * \rho_{\eps} \right) (-\frac{i}6 x,y) &= \int_{S_k} \rho_{\eps}(- \frac{i}6 x - s,t) \, d\mathcal{H}^1 = \int_{\tilde{S}^i_k} \rho_{\eps}^i(-x-s ,t) \, d\mathcal{H}^1 = \left( \mathcal{H}^1_{|\tilde{S}_k^i} * \rho^i_{\eps} \right)(-x,y),
    \end{align}
    where $\rho_{\eps}^i(s,t) = \frac{6}i\rho_{\eps}(\frac{i}6s,t)$, $\tilde{S}_k^i = \{ \tilde{x}_k^i + t e_1: t> 0\}$ and $\tilde{x}_k^i$ is the point whose $x$-coordinate is scaled by $\frac{i}6$ compared to $x_k$ while its $y$-coordinate is the same as the one of $x_k$.
    It follows that 
    \begin{equation} \label{eq: curl tilde beta}
    \operatorname{curl } \tilde{\beta} = \sigma \sum_{k=1}^n \left( \gamma_k \delta_{x_k} * \rho_{\eps} + \gamma_k \sum_{i=1}^3 \delta_{\tilde{x}_k^i} * \rho_{\eps}^i   \right).
    \end{equation}
    Then we estimate using Lemma \ref{lemma: est elliptic}
    \begin{equation}
    \int_{(0,1/2) \times (1/4,3/4)} |D\beta|^2 \, d\mathcal{L}^2 \leq C \left( \int_{(-1/4,3/4)\times (0,1)} (\partial_1  \tilde{\beta}_1)^2 + (\partial_2 \tilde{\beta}_2)^2 + 16  |\tilde{\beta}|^2 + |\operatorname{curl }\tilde{\beta}|^2  \, d\mathcal{L}^2 \right) \label{eq: appl elliptic}
    \end{equation}
    Firstly, we note that by the form of $\partial_i \tilde{\beta}_i$ it holds
    \[
    \int_{(-1,1) \times (0,1)} (\partial_1  \tilde{\beta}_1)^2 + (\partial_2 \tilde{\beta}_2)^2 \, d\mathcal{L}^2 \leq C \int_{(0,1)^2} (\partial_1  \beta_1)^2 + (\partial_2 \beta_2)^2 \, d\mathcal{L}^2.
    \]
    Secondly, we observe by \eqref{eq: curl tilde beta}, $|\rho_{\eps}|,|\rho^i_{\eps}| \leq C \eps^{-2} \chi_{B_{\eps}(0)}$ and  Lemma \ref{lem: lb vortex general} that 
    \begin{align*}
        \int_{(-1,1) \times (0,1)} |\operatorname{curl } \tilde{\beta}|^2 \, d\mathcal{L}^2 &\leq C \sum_{k=1}^n \int_{B_{\eps}(x_k)} \frac{\sigma^2}{\eps^4} \, d\mathcal{L}^2 \\
        &\leq C n \frac{\sigma^2}{\eps^2} \leq \frac{C}{\sigma^2} \sum_{k=1}^n \int_{B_{\eps}(x_k)} W(\beta) \, d\mathcal{L}^2 \leq \frac{C}{\sigma^2} \int_{(0,1)^2} W(\beta) \, d\mathcal{L}^2.
    \end{align*}
    Eventually, note that
    \begin{align}
        \int_{(-1,1) \times (0,1)} |\tilde{\beta}|^2 \, d\mathcal{L}^2 &\leq C \sum_{i=1}^3 \int_{(-1,0) \times (0,1)} |\beta(-\frac{i}6x,y)|^2 \, d\mathcal{L}^2 + \int_{(0,1)^2} |\beta|^2 \, d\mathcal{L}^2 \\
        &\leq C \int_{(0,1)^2} |\beta|^2 \,d \mathcal{L}^2 \leq C \int_{(0,1)^2} W(\beta) + 1 \, d\mathcal{L}^2,
    \end{align}
    where we used that by Lemma \ref{lemma: est W} $|\beta|^2 \leq 2\operatorname{dist}(\beta,K)^2 + 4 \leq 2 W(\beta) + 4$.
     Hence, with the last three estimates, we conclude from \eqref{eq: appl elliptic} for $0< \sigma \leq 1$ that
    \begin{align*}
    \int_{(0,1/2) \times (1/4,3/4)} W(\beta) + \sigma^2 |D\beta|^2 \, d\mathcal{L}^2 & \leq C \left(\int_{(0,1)^2} W(\beta) + \sigma^2 (\partial_1 \beta_1)^2 + \sigma^2 (\partial_2 \beta_2)^2 \, d\mathcal{L}^2 + \sigma^2 \right).
    \end{align*}
\end{proof}

Armed with the above estimates, we will now prove the lower bound for the anisotropic energy $E_{\sigma,\theta,\eps}^{(a)}$.

\begin{proposition}\label{prop: lower aniso}
    There exists $c>0$ such that it holds for all $\theta \in (0,1/2]$ and $\sigma \geq \sqrt{2} \pi \eps >0$ that
    \[
    \inf_{\beta \in \mathcal{A}_{\sigma,\theta,\eps}^{(a)}} E_{\sigma,\theta,\eps}^{(a)}(\beta) \geq c \min\left\{ \theta^2, \sigma \left(\frac{|\log \sigma|}{|\log \theta|} + 1 \right), \theta \frac{\sigma^3}{\eps^2} + \sigma \theta |\log \theta| \right\}.
    \]
\end{proposition}
\begin{proof}
Let $\beta \in \mathcal{A}_{\sigma,\theta,\eps}^{(a)}$.
First, assume that $0 < \sigma \leq c_1 \theta \leq 1$, where $0 < c_1 \leq 2 $ will be specified below.  By Lemma \ref{lemma: est anisotropic} we have
\begin{equation}
c \int_{(0,1/2) \times (1/4,3/4)} W(\beta) + \sigma^2 |D\beta|^2 \, d\mathcal{L}^2 \leq \int_{(0,1)^2} W(\beta) + \sigma^2 |\partial_1 \beta_1|^2 + \sigma^2|\partial_2 \beta_2|^2 \, d\mathcal{L}^2 + \sigma^2. \label{eq: tilde beta 1}
\end{equation}
Next, define $\tilde{\beta}: (0,1)^2 \to \R^2$ as $\tilde{\beta}(x,y) = \beta(x/2, y/2 + 1/4)$. By a simple change of variables, it follows that
\begin{equation}
\int_{(0,1)^2} W(\tilde{\beta}) + \sigma^2 |D\tilde{\beta}|^2 \, d\mathcal{L}^2 = \int_{(0,1/2) \times (1/4,3/4)} 4 W(\beta) + \sigma^2 |D\beta|^2 \, d\mathcal{L}^2. \label{eq: tilde beta 2}
\end{equation}
Moreover, $\operatorname{curl} \tilde{\beta}(x,y) = \frac12 \operatorname{curl} \beta(x/2,y/2+1/4) \in \mathcal{M}_{\sigma/2, \eps/2}$ and $\tilde{\beta}_2(0,y) = \beta_2(0,y/2 + 1/4) = 1-2\theta$. Hence, $\tilde{\beta} \in \mathcal{A}_{\theta,\sigma/2,\eps/2}$ and we obtain from Proposition \ref{prop: main lower bound} that 
\begin{align}
\int_{(0,1)^2} W(\tilde{\beta}) + \left(\frac{\sigma}2\right)^2 |D\beta|^2 \, d\mathcal{L}^2 &\geq c \min\left\{ \theta^2, \frac{\sigma}2 \left(\frac{|\log \sigma/2|}{|\log \theta|} + 1 \right), \theta \frac{\sigma^3}{2 \eps^2} + \theta \frac{\sigma}2 |\log \theta| \right\} \\
&\geq \frac{c}2 \min\left\{ \theta^2, \sigma \left(\frac{|\log \sigma|}{|\log \theta|} + 1 \right), \theta \frac{\sigma^3}{\eps^2} + \theta \sigma |\log \theta| \right\}. \label{eq: tilde beta 3}
\end{align}
Combining \eqref{eq: tilde beta 1}, \eqref{eq: tilde beta 2} and \eqref{eq: tilde beta 3} it follows that
\begin{equation}\label{eq: lb case 1}
\int_{(0,1)^2} W(\beta) + \sigma^2 |\partial_1 \beta_1|^2 + \sigma^2|\partial_2 \beta_2|^2 \, d\mathcal{L}^2 + \sigma^2 \geq \frac{c}2  \min\left\{ \theta^2, \sigma \left(\frac{|\log \sigma|}{|\log \theta|} + 1 \right), \theta \frac{\sigma^3}{\eps^2} + \theta \sigma |\log \theta| \right\}.
\end{equation}
Now fix $c_1 = \min\left\{\frac{\sqrt{c}}2, \frac{c}4,1 \right\}$ where $c>0$ is the universal constant in the estimate above. Then it follows for $0 < \sigma \leq c_1 \theta \leq 1$ that $\sigma^2 \leq \frac{c}4 \theta^2$, $\sigma^2 \leq \frac{c}4 \sigma \leq \frac{c}4 \sigma \left( \frac{|\log \sigma|}{|\log \theta|} + 1\right)$ and $\sigma^2 \leq \frac{c}4 \theta \sigma \leq \frac{c}4 \left( \theta \frac{\sigma^3}{\eps^2} + \theta \sigma |\log \theta| \right)$, i.e.,
\[
\sigma^2 \leq \frac{c}4 \min\left\{ \theta^2, \sigma \left( \frac{|\log \sigma|}{|\log \theta|} + 1\right), \theta \frac{\sigma^3}{\eps^2} + \theta \sigma |\log \theta|\right\}.
\]
Then the assertion follows from \eqref{eq: lb case 1}.

Next, we  will assume $\sigma \geq c_1 \theta$. 
By Remark \ref{rem: ani} it holds
\[
E_{\sigma,\theta,\eps}^{(a)}(\beta) \geq c_{LB}^{(2)} \min\left\{ \theta^2, \sigma, \theta \frac{\sigma^3}{\eps^2} \right\} \geq c_{LB}^{(2)}  \min\left\{ \theta^2, c_1 \theta, c_1 \theta^2 \frac{\sigma^2}{\eps^2} \right\} \geq c_{LB}^{(2)}  c_1 \theta^2.
\]
This proves the assertion in the case $\sigma \geq c_1 \theta$.

\end{proof}

\section*{Appendix}

Eventually, we provide an estimate that shows that in case of the energy $E_{\sigma,\theta,\eps}^{(1)}$ the estimate for the ball construction, Proposition \ref{prop: est annulus}, could be replaced by the simpler estimate below which only involves the regularizing term.
The same is not possible for the energy $E_{\sigma,\theta,\eps}^{(2)}$ since the typical behavior near a vortex at the origin is $|D \beta|^2 \sim \frac{\sigma}{|x|^2}$ and therefore $\sigma^2 \int_{B_1(0) \setminus B_{\eps}(0)} |D\beta|^2 \sim \frac{\sigma^4}{\eps^2} \ll \sigma^2 |\log \theta|$ for $\theta > 0$ small enough. 

\begin{lemma}
    For all $0 < R_1 < R_2$ and $\beta: B_{R_2}(0) \to \R^2$ with $\operatorname{curl } \beta = 0$ in $B_{R_2}(0) \setminus B_{R_1}(0)$ it holds 
    \[
    |D\beta|(B_{R_2}(0) \setminus B_{R_1}(0)) \geq  |\operatorname{curl } \beta(B_{R_1}(0)) | \, \log \frac{R_2}{R_1}.
    \]
\end{lemma}
\begin{proof}
    By approximation, we may assume that $\beta$ is smooth. 
    Since, $\operatorname{curl } \beta = 0$ in $B_{R_2}(0) \setminus B_{R_1}(0)$ it holds for $R_1 < r < R_2$ that
    \[
    \frac{d}{dr} \int_{\partial B_{r}(0)} \beta \cdot \tau \, d \mathcal{H}^1 = 0.
    \]
    On the other hand, we compute using the notation $x^{\perp} = (-x_2,x_1)^T$
    \begin{align*}
       0= &\frac{d}{dr} \int_{\partial B_{r}(0)} \beta \cdot \tau \, d \mathcal{H}^1 \\
        =&\frac{d}{dr} r \int_{\partial B_1(0)} \beta(rx) \cdot x^{\perp} \, d\mathcal{H}^1 \\
        =&\int_{\partial B_1(0)} \beta(rx) \cdot x^{\perp} \, d\mathcal{H}^1 + r \int_{\partial B_1(0)} (\partial_r\beta)(rx) \cdot x^{\perp} \, d\mathcal{H}^1 \\
        =& \frac{1}r \int_{\partial B_{r}(0)} \beta \cdot \tau \, d \mathcal{H}^1 +  \int_{\partial B_{r}(0)} (\partial_r\beta) \cdot \tau \, d \mathcal{H}^1 \\
        =& \frac{1}r \, (\operatorname{curl } \beta) (B_{R_1}(0)) +  \int_{\partial B_{r}(0)} (\partial_r\beta) \cdot \tau \, d \mathcal{H}^1.
    \end{align*}
    In particular, it follows
    \begin{align*}
    |D \beta|(B_{R_2}(0) \setminus B_{R_1}(0)) &\geq \int_{R_1}^{R_2} \int_{\partial B_r(0)} |\partial_r \beta| \, d\mathcal{H}^1 dr \\
    &\geq |\operatorname{curl } \beta( B_{R_1}(0))| \,  \int_{R_1}^{R_2} \frac1r \, dr = |\operatorname{curl } \beta( B_{R_1}(0))| \, \log \frac{R_2}{R_1}.
    \end{align*}
\end{proof}

\bibliographystyle{abbrv}
\bibliography{main.bib}
    
\end{document}